\documentclass{amsart}
\title{A model of second-order arithmetic satisfying AC but not DC}
\author{Sy-David Friedman}
\address[S.-D. Friedman] {Kurt \Godel\ Research Center, University of Vienna, W\"ahringer Strasse 25 A-1090 Vienna, Austria}
\email{sdf@logic.univie.ac.at}
\urladdr{http://www.logic.univie.ac.at/~sdf}
\thanks{The first author wishes to acknowledge the support of the Austrian Science Fund (FWF) through Research Project I 1238.}

\author{Victoria Gitman}
\address[V. Gitman]{The City University of New York, CUNY Graduate Center, Mathematics Program, 365 Fifth Avenue, New York, NY 10016}
\email{vgitman@nylogic.org}
\urladdr{https://victoriagitman.github.io}

\author{Vladimir Kanovei}
\address[V. Kanovei]{Institute for the Information Transmission Problems (IITP), 19, build.1, Bolshoy Karetny per. Moscow 127051, Russia and
Russian University of Transport (RUT-MIIT), Moscow, Russia}
\email{kanovei@googlemail.com}
\thanks{The third author wishes to acknowledge the support of the RFBR grant 17-01-00705.}

\usepackage{latexsym}
\usepackage{amssymb, latexsym}
\usepackage{amsmath}
\usepackage{mathrsfs}
\usepackage{amsxtra}
\usepackage{amsthm}
\usepackage{verbatim}
\newtheorem{theorem}{Theorem}[section]

\newtheorem{lemma}[theorem]{Lemma}

\newtheorem{proposition}[theorem]{Proposition}

\theoremstyle{definition}
\newtheorem{definition}[theorem]{Definition}
\newtheorem{question}[theorem]{Question}

\newcommand{\image}{\mathbin{\hbox{\tt\char'42}}}

\newcommand{\Union}{\bigcup}
\newcommand{\union}{\cup}

\newcommand{\lt}[1]{{\smalllt}#1}

\newcommand{\smallleq}{\mathrel{\mathchoice{\raise2pt\hbox{$\scriptstyle\leq$}}{\raise1pt\hbox{$\scriptstyle\leq$}}{\raise1pt\hbox{$\scriptscriptstyle\leq$}}{\scriptscriptstyle\leq}}}
\newcommand{\smalllt}{\mathrel{\mathchoice{\raise2pt\hbox{$\scriptstyle<$}}{\raise1pt\hbox{$\scriptstyle<$}}{\raise0pt\hbox{$\scriptscriptstyle<$}}{\scriptscriptstyle<}}}

\newcommand{\ZFC}{{\rm ZFC}}
\newcommand{\one}{\mathop{1\hskip-2.5pt {\rm l}}}
\newcommand{\p}{\mathbb{P}}
\newcommand{\q}{\mathbb{Q}}
\newcommand{\U}{\mathbb{U}}
\newcommand{\la}{\langle}
\newcommand{\ra}{\rangle}
\newcommand{\forces}{\Vdash}
\newcommand{\restrict}{\upharpoonright}
\newcommand{\ex}[1]{{}^{#1}}
\newcommand{\T}{\mathcal T}
\newcommand{\AC}{{\rm AC}}
\newcommand{\DC}{{\rm DC}}
\newcommand{\ZF}{{\rm ZF}}
\newcommand{\Godel}{G\"odel}
\newcommand{\sym}{\text{sym}}
\newcommand{\HS}{{\rm HS}}
\newcommand{\Add}{{\rm Add}}
\newcommand{\lev}{{\rm lev}}
\newcommand{\concat}{\mathbin{{}^\smallfrown}}
\newcommand{\Aut}{{\rm Aut}}
\newcommand{\dom}{\text{dom}}
\newcommand{\M}{\mathcal M}

\newcommand{\Z}{{\rm Z}}
\newcommand{\Levy}{L\'{e}vy}
\begin{document}
\maketitle
\begin{abstract}
We show that there is a $\beta$-model of second-order arithmetic in which the choice scheme holds, but the dependent choice scheme fails for a $\Pi^1_2$-assertion, confirming a conjecture of Stephen Simpson. We obtain as a corollary that the Reflection Principle, stating that every formula reflects to a transitive set, can fail in models of $\ZFC^-$. This work is a rediscovery by the first two authors of a result obtained by the third author in \cite{kanovei:ACnotDC}.
\end{abstract}
\section{Introduction}
Models of arithmetic are two-sorted structures, having two types of objects, which we think of as numbers and sets of numbers. Their properties are formalized using a two-sorted logic with separate variables and quantifiers for numbers and sets. By convention, we will denote number variables by lower-case letters and sets variables by upper-case letters. The language of second-order arithmetic is the language of first-order arithmetic $\mathcal L_A=\{+,\cdot,<,0,1\}$ together with a membership relation $\in$ between numbers and sets. A multitude of second-order arithmetic theories, as well as the relationships between them, have been extensively studied (see \cite{simpson:second-orderArithmetic}).

An example of a weak second-order arithmetic theory is ${\rm ACA_0}$, whose axioms consist of the modified Peano axioms, where instead of the induction scheme we have the single second-order induction axiom $$\forall X [(0\in X\wedge \forall n(n\in X\rightarrow n+1\in X))\rightarrow \forall n (n\in X)],$$ and the \emph{comprehension scheme} for first-order formulas. The latter is a scheme of assertions stating for every first-order formula, possibly with set parameters, that there is a set whose elements are exactly the numbers satisfying the formula. One of the strongest second-order arithmetic theories is $\Z_2$, often referred to as \emph{full second-order arithmetic}, which strengthens comprehension for first-order formulas in ${\rm ACA}_0$ to full comprehension for all second-order assertions. This means that for a formula with any number of second-order quantifiers, there is a set whose elements are exactly the numbers satisfying the formula. For example, the reals of any model of $\ZF$ give a model of $\Z_2$. We can further strengthen the theory $\Z_2$ by adding choice principles for sets: the choice scheme and the dependent choice scheme.

The \emph{choice scheme} is a scheme of assertions, which states for every second-order formula $\varphi(n,X,A)$ with a set parameter $A$ that if for every number $n$, there is a set $X$ witnessing $\varphi(n,X,A)$, then there is a single set $Y$ collecting witnesses for every $n$, in the sense that $\varphi(n,Y_n,A)$ holds, where $Y_n=\{m\mid \la n,m\ra\in Y\}$ and $\la n,m\ra$ is any standard coding of pairs. More precisely, an instance of the choice scheme for the formula $\varphi(n,X,A)$ is $$\forall n\exists X\,\varphi(n,X,A)\rightarrow \exists Y\forall n\,\varphi(n,Y_n,A).$$ We will denote by $\Sigma^1_n$-$\AC$ the fragment of the choice scheme for $\Sigma^1_n$-assertions, making an analogous definition for $\Pi^1_n$, and we will denote the full choice scheme by $\Sigma^1_\infty$-$\AC$. The reals of any model of $\ZF+\AC_\omega$ (countable choice) satisfy $\Z_2+\Sigma^1_\infty$-$\AC$. It is a folklore result, going back possibly to Mostowski, that the theory $\Z_2+\Sigma^1_\infty$-$\AC$ is bi-interpretable with the theory $\ZFC^-$ ($\ZFC$ without the powerset axiom, with Collection instead of Replacement) together with the statement that every set is countable.

The \emph{dependent choice scheme} is a scheme of assertions, which states for every second-order formula $\varphi(X,Y,A)$ with set parameter $A$ that if for every set $X$, there is a set $Y$ witnessing $\varphi(X,Y,A)$, then there is a single set $Z$ making infinitely many dependent choices according to $\varphi$. More precisely, an instance of the dependent choice scheme for the formula $\varphi(X,Y,A)$ is $$\forall X\exists Y\,\varphi(X,Y,A)\rightarrow \exists Z\forall n\,\varphi(Z_n,Z_{n+1},A).$$ We will denote by $\Sigma^1_n$-$\DC$ the dependent choice scheme for $\Sigma^1_n$-assertions, with an analogous definition for $\Pi^1_n$, and we will denote the full dependent choice scheme by $\Sigma^1_\infty$-$\DC$. The reals of a model of $\ZF+\DC$ (dependent choice) satisfy $\Z_2+\Sigma^1_\infty$-$\DC$.

It is not difficult to see that the theory $\Z_2$ implies $\Sigma^1_2$-$\AC$, the choice scheme for $\Sigma^1_2$-assertions. Models of $\Z_2$ can build their own version of \Godel's constructible universe $L$. If a model of $\Z_2$ believes that a set $\Gamma$ is a well-order, then it has a set coding a set-theoretic structure constructed like $L$ along the well-order $\Gamma$.  It turns out that models of $\Z_2$ satisfy a version of Shoenfield's absoluteness with respect to their constructible universes. For every $\Sigma^1_2$-assertion $\varphi$, a model of $\Z_2$ satisfies $\varphi$ if and only its constructible universe satisfies $\varphi$ with set quantifiers naturally interpreted as ranging over the reals. All of the above generalizes to constructible universes $L[A]$ relativized to a set parameter $A$.\footnote{Indeed, the much weaker theory ${\rm ATR}_0$ suffices for everything we have said so far. See \cite{simpson:second-orderArithmetic} (Section VII.3 and VII.4) for details.} Thus, given a $\Sigma^1_2$-assertion $\varphi(n,X,A)$ for which the model satisfies $\forall n\exists X\,\varphi(n,X,A)$, the model can go to its constructible universe $L[A]$ to pick the least witness $X$ for $\varphi(n,X,A)$ for every $n$, because $L[A]$ agrees when $\varphi$ is satisfied, and then put the witnesses together into a single set using comprehension. So long as the unique witnessing set can be obtained for each $n$, comprehension suffices to obtain a single set of witnesses. How much more of the choice scheme follows from $\Z_2$? The reals of the classical Feferman-\Levy\ model of $\ZF$ (see \cite{levy:choicescheme}, Theorem 8), in which $\aleph_1$ is a countable union of countable sets, is a $\beta$-model of $\Z_2$ in which $\Pi^1_2$-$\AC$ fails. This is a particulary strong failure of the choice scheme because, as we explain below, $\beta$-models are meant to strongly resemble the full standard model given by $P(\omega)$.

There are two ways in which a model of second-order arithmetic can resemble the full standard model given by $P(\omega)$. A model of second-order arithmetic is called an $\omega$-\emph{model} if its first-order part is $\omega$, and it follows that its second-order part is some subset of $P(\omega)$. But even an $\omega$-model can poorly resemble $P(\omega)$ because it may be wrong about well-foundedness by missing $\omega$-sequences. An $\omega$-model of second-order arithmetic which is correct about well-foundedness is called a $\beta$-\emph{model}. The reals of any transitive $\ZF$-model is a $\beta$-model of $\Z_2$. One advantage to having a $\beta$-model of $\Z_2$ is that the constructible universe it builds internally is isomorphic to an initial segment $L_\alpha$ of the actual constructible universe $L$.


The theory $\Z_2$ also implies $\Sigma^1_2$-$\DC$ (see \cite{simpson:second-orderArithmetic}, Theorem VII.9.2), the dependent choice scheme for $\Sigma^1_2$-assertions.  In this article, we construct a symmetric submodel of a forcing extension of $L$ whose reals form a model of second-order arithmetic in which $\Z_2$ together with $\Sigma^1_\infty$-$\AC$ holds, but $\Pi^1_2$-$\DC$ fails. The forcing notion we use is a tree iteration of Jensen's forcing for adding a unique generic real.

Jensen's forcing, which we will call here $\p^J$, introduced by Jensen in \cite{jensen:real}, is a subposet of Sacks forcing constructed in $L$ using the $\diamondsuit$ principle. The poset $\p^J$ has the ccc and adds a unique generic real over $L$. The collection of all $L$-generic reals for $\p^J$ in any model is $\Pi^1_2$-definable. Jensen used his forcing to show that it is consistent with $\ZFC$ that there is a $\Pi^1_2$-definable non-constructible real singleton \cite{jensen:real}. Recently Lyubetsky and the third author extended the ``uniqueness of generic filters" property of Jensen's forcing to finite-support products of $\p^J$ \cite{kanovei:productOfJensenReals}. They showed that in a forcing extension $L[G]$ by the $\omega$-length finite support-product of $\p^J$, the only $L$-generic reals for $\p^J$ are the slices of the generic filter $G$. The result easily extends to $\omega_1$-length finite support-products as well.

We in turn extend the ``uniqueness of generic filters" property to tree iterations of Jensen's forcing. We first define finite iterations $\p^J_n$ of Jensen's forcing $\p^J$, and then define an iteration of $\p^J$ along a tree $\T$ to be a forcing whose conditions are functions from a finite subtree of $\T$ into $\Union_{n<\omega}\p_n^J$ such that nodes on level $n$ get mapped to elements of the $n$-length iteration $\p_n^J$ and conditions on higher nodes extend conditions on lower nodes. The functions are ordered by extension of domain and strengthening on each coordinate. We show that in a forcing extension $L[G]$ by the tree iteration of $\p^J$ along the tree $\ex{\lt\omega}\omega_1$ (or the tree $\ex{\lt\omega}\omega$) the only $L$-generic filters for $\p_n^J$ are the restrictions of $G$ to level $n$ nodes of tree. We proceed to construct a symmetric submodel of $L[G]$ which has the tree of $\p_n^J$-generic filters added by $G$ but no branch through it. The symmetric model we construct satisfies $\AC_\omega$ and the tree of $\p_n^J$-generic filters is $\Pi^1_2$-definable in it. The reals of this model thus provide the desired $\beta$-model of $\Z_2$ in which $\Sigma^1_\infty$-$\AC$ holds, but $\Pi^1_2$-$\DC$ fails.
\begin{theorem}\label{th:main1}
There is a $\beta$-model of second-order arithmetic $\Z_2$ together with $\Sigma^1_\infty$-$\AC$ in which $\Pi^1_2$-$\DC$ fails.
\end{theorem}
\noindent It should be noted that in our model the instance of $\Pi^1_2$-$\DC$ failure is parameter-free.

Theorem \ref{th:main1} has a long and complicated history. In 1973, Simpson submitted an abstract claiming a proof of the result \cite{simpson:ACnotDC}, but he didn't follow it up with a publication because, as he kindly shared with the first author, he never worked out the details of the argument, which involved Jensen's forcing, but did not have all the parts needed to solve the problem. In 1979, the third author published, in Russian, a technical report with the result \cite{kanovei:ACnotDC}. Finally, after the third author's (joint with Lyubetsky) recent publications on uniqueness properties of Jensen's forcing, the first and second author independently rediscovered the third author's result, leading to this joint publication.

Our results also answer a long-standing open question of Zarach from \cite{Zarach1996:ReplacmentDoesNotImplyCollection} about whether the Reflection Principle holds in models of $\ZFC^-$. The \emph{Reflection Principle} states that every formula can be reflected to a transitive set, and holds in $\ZFC$ by the \Levy-Montague reflection because every formula is reflected by some $V_\alpha$. In the absence of the von Neumann hierarchy, it is not clear how to realize reflection, and indeed we show that it fails in $H_{\omega_1}\models\ZFC^-$ of the symmetric model we construct.

\begin{theorem}\label{th:main2}
The theory $\ZFC^-$ does not imply the Reflection Principle.
\end{theorem}
\section{Jensen's forcing}\label{sec:PerfectPosets}
Jensen's poset $\p^J$, introduced in \cite{jensen:real}, is a subposet of Sacks forcing in $L$ with the countable chain condition and the property that it adds a unique generic real. The poset $\p^J$ is constructed in $L$ as the union of a continuous chain $\la \p_\alpha\mid\alpha<\omega_1\ra$ of length $\omega_1$ of countable perfect posets. At successor stages $\alpha+1$, the principle $\diamondsuit$ is used to seal a certain countable collection of maximal antichains of $\p_\alpha$, so that by the end of the construction every maximal antichain of $\p^J$ is sealed and therefore countable.

Recall that a tree $T\subseteq \ex{\lt\omega}2$ is \emph{perfect} if every node in $T$ has a splitting node above it. Given two perfect trees $T$ and $S$, let's define their \emph{meet} $T\wedge S$ as follows. Let $U_0=T\cap S$, and, assuming $U_\alpha$ has been defined, let $U_{\alpha+1}$ be the set of all nodes in $U_\alpha$ which have a splitting node in $U_\alpha$ above them. At limit stages, take intersections. Since $U_0$ is countable, there must be a countable stage $\alpha$ in the construction for which $U_\alpha=U_{\alpha+1}$. We define $T\wedge S=U_\alpha$. It is not difficult to see that either $T\wedge S=\emptyset$ or $T\wedge S$ is a perfect tree. In the latter case, it is the maximal perfect tree contained in $T\cap S$, so that every perfect tree $U\subseteq T\cap S$ is contained in $T\wedge S$. Recall the standard terminology that if $T\subseteq \ex{\lt\omega}2$ is a tree and $s$ is a node in $T$, then $T_s$ denotes the subtree of $T$ consisting of all nodes compatible with $s$. The following proposition is straightforward to check.
\begin{proposition} If $T,S$, and $R$ are perfect trees, then
\begin{enumerate}
\item $(T\wedge S)_s=T_s\wedge S_s$ for every node $s\in T\wedge S$,
\item $(T\wedge S)\wedge R=T\wedge (S\wedge R)$ ($\wedge$ is associative),
\item $(T\union S)\wedge R=(T\wedge R)\union (S\wedge R)$ ($\wedge$ distributes over $\union$).
\end{enumerate}
\end{proposition}
Let's recall now some facts about perfect posets, which are subposets of Sacks forcing closed under certain basic operations.
\begin{definition}
We say that a collection $\p$ of perfect trees ordered by inclusion is a \emph{perfect poset}\footnote{Note that this is not a standard definition. In the literature, a \emph{perfect poset} is usually defined to be a collection $\p$ of perfect trees such that $\ex{\lt\omega}2\in\p$ and whenever $T\in\p$ and $s\in T$, then $T_s\in\p$.} if
\begin{enumerate}
\item $(\ex{\lt\omega}2)_s\in\p$ for every $s\in \ex{\lt\omega}2$, and\\
for every $T,S\in \p$,
\item $T\union S\in \p$ (closed under unions),
\item $T\wedge S\in \p$, if $T\wedge S\neq\emptyset$ (closed under meets).
\end{enumerate}
\end{definition}
\noindent The smallest perfect poset, which we will denote by $\p^{\text{min}}$, is the closure under finite unions of the collection $\{(\ex{\lt\omega}2)_s\mid s\in\ex{\lt\omega}2\}$. Note that two perfect trees $T$ and $S$ are compatible under the inclusion ordering precisely when $T\wedge S\neq\emptyset$. So if $\p\subseteq \q$ are perfect posets, then $T,S\in\p$ are compatible in $\q$ if and only if they are already compatible in $\p$. Standard arguments show that a generic filter $G$ for a perfect poset $\p$ is determined by the \emph{generic real} $r$ that  is a branch through $T\in\p$ if and only if $T\in G$.

Given a perfect poset $\p$, we will denote by $\p^{\lt\omega}$ the $\omega$-length finite-support product of $\p$. Conditions in $\p^{\lt\omega}$ are functions $p:\omega\to \p$ such that for all but finitely many $n$, $p(n)=\ex{\lt\omega}2$. We will sometimes abuse notation by writing $p=\la T_0,\ldots,T_{n-1}\ra$ for conditions in $\p^{\lt\omega}$, meaning that all remaining coordinates are trivial.

Following \cite{abraham:jensenRealsIterations}, we associate to a perfect poset $\p$, the poset $\q(\p)$ whose conditions are pairs $(T,n)$ with $T\in\p$ and $n\in\omega$ ordered so that $(T_2,n_2)\leq (T_1,n_1)$ whenever $T_2\subseteq T_1$, $n_2\geq n_1$, and $T_1\cap \ex{n_1}2=T_2\cap \ex{n_1}2$. We will refer to $\q(\p)$ as the \emph{fusion poset} for $\p$ because fusion arguments involving $\p$ amount to producing a filter for $\q(\p)$ meeting sufficiently many dense sets. If $G\subseteq \q(\p)$ is $V$-generic, then the union of $T\cap \ex{n}2$ for $(T,n)\in G$ is the generic perfect tree $\T$ added by $\q(\p)$. Note that $\T\leq T$ for every $T$ that appears in some condition in $G$.

Let $\q(\p)^{\lt\omega}$ denote the $\omega$-length finite-support product of the $\q(\p)$. Conditions in $\q(\p)^{\lt\omega}$ are function $q:\omega\to \q(\p)$ such that for all but finitely many $n$, $q(n)=(\ex{\lt\omega}2,0)$, but we will sometimes abuse notation by writing the conditions as finite tuples.

Suppose that $G\subseteq \q(\p)^{\lt\omega}$ is $V$-generic. Let $\la \T_n\mid n<\omega\ra\in V[G]$ be the $\omega$-length sequence of generic perfect trees derived from $G$ and let $\dot{\T}_n$ be the canonical $\q(\p)^{\lt\omega}$-names for the trees $\T_n$. In $V[G]$, let $$\U=\{\T_n\wedge S\mid S\in\p,\,n\in\omega,\,\T_n\wedge S\neq\emptyset\}.$$
We will denote by $\p^*$ the collection of perfect trees that is the closure under finite unions of $\p$ and $\U$ and argue that $\p^*$ is a perfect poset extending $\p$. This will be the case even when $G$ is only generic for a countable transitive model $M$ containing $\p$. At successor stages $\alpha+1$ in the construction of Jensen's poset $\p^J$, if $\diamondsuit$ codes a certain countable transitive model $M$ with $\p_\alpha\in M$, we will extend $\p_\alpha$ to $\p_{\alpha+1}=\p_\alpha^*$ constructed in a generic extension $M[G]$ by $\q(\p_\alpha)^{\lt\omega}$.

Suppose that $M$ is a countable transitive model of $\ZFC^-+``\mathcal P(\omega)$ exists" and $\p\in M$ is a perfect poset. Clearly we have $\q(\p)\in M$. We will argue that if $G\subseteq\q(\p)^{\lt\omega}$ is $M$-generic, then the poset $\p^*$ constructed in $M[G]$ is a perfect poset with the property that every maximal antichain of $\p$ that is an element of $M$ remains maximal in $\p^*$. First, we need the following easy proposition.

\begin{proposition}\label{prop:subtreeContainedInS}
Suppose that $\p$ is a perfect poset and a condition $q\in\q(\p)^{\lt\omega}$ forces that $\dot{\T}_n\wedge \check S\neq\emptyset$ with $S\in \p$. Then there is a condition $\bar q\leq q$ with $\bar q(n)=(\bar T,m)$ and a node $s\in \bar T\cap \ex{m}2$ for which $(\bar T)_s\leq S$.
\end{proposition}
\begin{proof}
Let $q(n)=(T,m)$. Since $q\forces \dot{\T_n}\wedge \check S\neq\emptyset$, there must be some $s\in T\cap \ex{m}2$ such that $U=T_s\wedge S\neq\emptyset$. Let $\bar T$ be the perfect tree we get by replacing $T_s$ with $U$ in $T$ and let $\bar q\leq q$ be such that $\bar q(n)=(\bar T,m)$ and $\bar q(i)=q(i)$ for all $i\neq n$. Note that the closure under unions property of perfect posets is needed to conclude that $\bar T\in\p$.
\end{proof}

\begin{proposition}\label{prop:propertiesOfP*}
Suppose $M$ is a countable transitive model of $\ZFC^-+``\mathcal P(\omega)$ exists" and $\p\in M$ is a perfect poset. If $G\subseteq \q(\p)^{\lt\omega}$ is $M$-generic and $\p^*$ is constructed in $M[G]$ as above, then:
\begin{enumerate}
\item $\p^*$ is a perfect poset.
\item $\U$ is dense in $\p^*$ and $\{\T_n\mid n<\omega\}$ is a maximal antichain of $\p^*$.
\item Every maximal antichain of $\p$ from $M$ remains maximal in $\p^*$.
\end{enumerate}
\end{proposition}
\begin{proof}
First, let's prove (2). We argue that every $T\in\p$ has some generic tree $\T_n$ below it. Fix $T\in\p$. Let $q\in\q(\p)^{\lt\omega}$. Since $q$ has finite support, we can choose some $n$ such that $q(n)=(\ex{\lt\omega}2,0)$ and strengthen $q$ to a condition $\bar q$ such that $q(m)=\bar q(m)$ for all $m\neq n$ and $\bar q(n)=(T,0)$. Clearly $\bar q\forces \dot{\T}_n\leq T$. So by density, there must be some such $\bar q\in G$.

Next, let's show that for $i\neq j$, $\T_i\wedge \T_j=\emptyset$, and so the generic trees $\T_n$ form an antichain. Fix any condition $p\in\q(\p)^{\lt\omega}$ with $p(n)=(T_n,m_n)$. By strengthening further, we can assume that $m_i=m_j=m$. For each node $s$ that is on level $m$ of both $T_i$ and $T_j$, we can choose disjoint perfect trees $U^{s,i}\subseteq (T_i)_s$ and $U^{s,j}\subseteq (T_j)_s$. Let $S_i$ be the perfect tree which we get by replacing, for each $s$ on level $m$ of both $T_i$ and $T_j$, $(T_i)_s$ with $U^{s,i}$ in $T_i$, and let $S_j$ be obtained similarly. Now let $\bar p\leq p$ be the condition where on coordinate $i$, we put $S_i$  instead of $T_i$ and on coordinate $j$, we put $S_j$ instead of $T_j$. Clearly $\bar p$ forces that $\T_i\wedge\T_j=\emptyset$. It remains to show that every tree $T$ in $\p^*$ is compatible with some $\T_n$. If $T$ is a finite union of trees one of which is in $U$, then this is clear, and otherwise $T\in\p$, and we just argued that then some $\T_n\leq T$.

To prove (1), it suffices to show that $\p^*$ is closed under meets. So suppose that $R=\Union_{i<n} R_i$ and $R'=\Union_{j<m} R'_j$, where $R_i,R'_j\in\p\union \U$, and consider $R\wedge R'=\Union_{i<n,j<m}R_i\wedge R'_j$. We will argue that each $R_i\wedge R'_j$ is either empty or in $\p\union \U$. If both $R_i$ and $R'_j$ are in $\p$, then the conclusion follows since $\p$ is a perfect poset. If one of the $R_i$ or $R'_j$ is $\T_k\wedge S$ with $S\in\p$ and the other is $S'\in\p$, then $R_i\wedge R'_j=\T_k\wedge (S\wedge S')$ and the conclusion follows. Finally if $R_i=\T_k\wedge S$ and $R'_j=\T_{k'}\wedge S'$ with $S,S'\in\p$, then $R_i\wedge R'_j=(\T_k\wedge \T_{k'})\wedge (S\wedge S')$. If $k=k'$, then the conclusion follows and if $k\neq k'$, we just argued above that $\T_k\wedge\T_{k'}=\emptyset$.

Let's now prove (3). Fix a maximal antichain $\mathcal A\in M$ of $\p$. It suffices to show that every tree in $\U$ is compatible with some element of $\mathcal A$. So suppose that $\T_n\wedge S\in \U$ with $S\in\p$, and fix a condition $q\in G$ forcing that $\dot{\T}_n\wedge \check S\neq \emptyset$. We will argue that there is a a condition $q'\leq q$ with $q'(n)=(T',m)$ such that there is a node $s\in T'\cap \ex{m}2$ and $A\in\mathcal A$ for which $T'_s\leq S,A$. By Proposition~\ref{prop:subtreeContainedInS}, there is $\bar q\leq q$ with $\bar q(n)=(\bar T,m)$ such that there is $s\in \ex{m}2$ with $\bar T_s\leq S$. Since $\mathcal A$ is maximal in $\p$, there is $A\in\mathcal A$ and $U\in\p$ such that $U\leq \bar T_s,A$. Let $T'$ be the tree we get by replacing $\bar T_s$ with $U$ in $\bar T$ and let $q'\leq q$ be the condition where $q'(k)=\bar q(k)$ for all $k\neq n$ and $q'(n)=(T',m)$. By density, there is some such condition $q'\in G$. But this means that $(\T_n)_s\leq \T_n\wedge S$ and $(\T_n)_s\leq A$, so that $\T_n\wedge S$ is compatible with $A\in\mathcal A$.
\end{proof}
We let $\U^{\lt\omega}$ denote the collection of elements $p\in\p^{*\lt\omega}$ such that for all $n$, $p(n)\in \U$ or $p(n)=\ex{\lt\omega}2$. Note that it follows from Proposition~\ref{prop:propertiesOfP*} that  $\U^{\lt\omega}$ is dense in $\p^{*\lt\omega}$.

\begin{proposition}\label{prop:propertiesOfProductP*}
Suppose that $M$ is a countable transitive model of $\ZFC^-+``\mathcal P(\omega)$ exists" and $\p\in M$ is a perfect poset. If $G\subseteq \q(\p)^{\lt\omega}$ is $M$-generic and $\p^*$ is constructed in $M[G]$ as above, then every maximal antichain of $\p^{\lt\omega}$ from $M$ remains maximal in $\p^{*\lt\omega}$.
\end{proposition}
\begin{proof}
Fix a maximal antichain $\mathcal A\in M$ of $\p^{\lt\omega}$. It suffices to show that every condition $p\in\U^{\lt\omega}$ is compatible with some element of $\mathcal A$. Let $$p=\la \T_{i_0}\wedge S_0,\T_{i_1}\wedge S_1,\ldots,\T_{i_m}\wedge S_m\ra.$$ Note that the values $i_n$ do not have to be distinct and we need to carefully address this possibility. Fix a condition $q\in G$ forcing that $\dot \T_{i_n}\wedge \check S_n\neq \emptyset$ for every $n\leq m$. Repeatedly using the construction in the proof of Proposition~\ref{prop:subtreeContainedInS} and going to a large enough level, we can find a condition $\bar q\leq q$ with $\bar q(n)=(\bar T_n,k_n)$ such that for every $n\leq m$, there is a node $s_n$ on level $k_{i_n}$ of $\bar T_{i_n}$ such that $(\bar T_{i_n})_{s_n}\leq S_n$, and for $n\neq n'$, $s_n\neq s_{n'}$. Let $\bar p=\la (\bar T_{i_0})_{s_0}, (\bar T_{i_1})_{s_1},\ldots,(\bar T_{i_m})_{s_m}\ra$. Since $\mathcal A$ is maximal in $\p^{\lt\omega}$, the condition $\bar p$ is compatible with some $a\in\mathcal A$, and so for every $n\leq m$, we can let $U_n=\bar p(n)\wedge a(n)\in \p$. Now, for every $i\in \{i_n\mid n\leq m\}$, let $R_i$ be the tree we get by replacing $(\bar T_i)_{s_n}$ with $U_n$ in $\bar T_i$ whenever $i_n=i$. Let $r$ be the condition such that $r(i)=(R_i,k_i)$ for $i\in \{i_n\mid n\leq m\}$ and $r(i)=\bar q(i)$ otherwise. By density, there is some such $r\in G$. But this means that $(\T_{i_n})_{s_n}\leq S_n\wedge a(n)$, so that $p$ is compatible with $a$.
\end{proof}

All our constructions will take place inside the constructible universe $L$. For reasons that will become obvious during the course of the constructions, we will restrict ourselves to countable models of $\ZFC^-+``\mathcal P(\omega)$ exists" which happen to be initial segments $L_\alpha$ of $L$. So let's call countable $L_\alpha$ satisfying $\ZFC^-+``\mathcal P(\omega)$ exists" \emph{suitable} models. Relevant examples of suitable models for us will be transitive collapses of countable $M\prec L_{\omega_2}$.

We are now ready to review the construction of Jensen's perfect poset $\p^J$, which will use the $\diamondsuit$-principle to anticipate and seal maximal antichains. So let's start by fixing a canonically defined $\diamondsuit$-sequence $\la S_\alpha\mid\alpha<\omega_1\ra$. Note that if $M$ is a suitable model and $\delta=\omega_1^M$, then $\la S_\alpha\mid \alpha<\delta\ra$ is an element of $M$.

Jensen's poset $\p^J$ will be the union of the following increasing sequence $\la \p_\alpha\mid\alpha<\omega_1\ra$ of perfect posets. Let $\p_0=\p^{\text{min}}$. At limit stages, we will take unions. Suppose $\p_\alpha$ has been defined. We let $\p_{\alpha+1}=\p_\alpha$, unless the following happens. Suppose $S_\alpha$ codes a well-founded and extensional binary relation $E\subseteq \alpha\times\alpha$ such that the collapse of $E$ is a suitable model $M_\alpha$ with $\p_\alpha\in M_\alpha$ and $\alpha=\omega_1^{M_\alpha}$. In this case, we take the $L$-least $M_\alpha$-generic filter $G\subseteq \q(\p_\alpha)^{\lt\omega}$ and let $\p_{\alpha+1}=\p^*_\alpha$ as constructed in $M_\alpha[G]$.

As we observed in Proposition~\ref{prop:propertiesOfP*}, $\p_{\alpha+1}$ is a perfect poset with the property that every maximal antichain of $\p_\alpha$ in $M_\alpha$ remains maximal in $\p_{\alpha+1}$. Also, by Proposition~\ref{prop:propertiesOfProductP*}, every maximal antichain of $\p_\alpha^{\lt\omega}$ in $M_\alpha$ remains maximal in $\p_{\alpha+1}^{\lt\omega}$. Now let's argue that every maximal antichain of $\p_\alpha^{\lt\omega}$ in $M_\alpha$ remains maximal in the final poset $\p^{J\lt\omega}$. It suffices to argue that the models $M_\alpha$ form an increasing sequence, which follows because if $\beta>\alpha$, then $\beta=\omega_1^{M_\beta}$, and therefore $M_\beta$ has $S_\alpha$ as an element and can collapse it to obtain $M_\alpha$. This shows that every maximal antichain of $\p_\alpha^{\lt\omega}$ that is an element of $M_\alpha$ is sealed.
\begin{theorem}\label{th:JensenForcingCCC}
The finite-support product $\p^{J\lt\omega}$ of Jensen's poset $\p^J$ has the ccc.
\end{theorem}
\begin{proof}
Fix a maximal antichain $\mathcal A$ of $\p^{J\lt\omega}$. Choose some transitive $M\prec L_{\omega_2}$ of size $\omega_1$ with $\mathcal A\in M$. We can decompose $M$ as the union of a continuous elementary chain of countable substructures $$X_0\prec X_1\prec\cdots\prec X_\alpha\prec\cdots\prec M$$ with $\mathcal A\in X_0$. By properties of $\diamondsuit$, there is some $\alpha$ such that $\alpha=\omega_1\cap {X_\alpha}$, $\p_\alpha=\p^J\cap X_\alpha$, and $S_\alpha$ codes $X_\alpha$. Let $M_\alpha$ be the transitive collapse of $X_\alpha$. Then $\p_\alpha$ is the image of $\p^J$ under the collapse and $\alpha$ is the image of $\omega_1$. Let $\bar {\mathcal A}=\mathcal A\cap X_\alpha$ be the image of $\mathcal A$ under the collapse. So at stage $\alpha$ in the construction of $\p^J$, we chose a forcing extension $M_\alpha[G]$ of $M_\alpha$ by $\q(\p_\alpha)^{\lt\omega}$ and let $\p_{\alpha+1}=\p^*_\alpha$ as constructed in $M_\alpha[G]$. Thus, by our observation above, $\bar {\mathcal A}$ remains maximal in $\p^{J\lt\omega}$, and hence $\bar {\mathcal A}=\mathcal A$ is countable.
\end{proof}
Finally, we would like to observe that every maximal antichain of $\p_{\alpha+1}$ from $M_\alpha[G]$, in particular the antichain $\la \T_n\mid n<\omega\ra$ of generic perfect trees, remains maximal in $\p^J$.
\begin{proposition}\label{prop:MaxAntiMForcingExtension}
Suppose that $\mathcal A\in M_\alpha[G]$ is a maximal antichain of $\p_{\alpha+1}$ (or $\p_{\alpha+1}^{\lt\omega})$. Then $\mathcal A$ remains maximal in $\p^J$ (or $\p^{J\lt\omega})$.
\end{proposition}
\begin{proof}
Fix $\beta>\alpha$. We already argued that $M_\alpha\in M_\beta$. Since $M_\beta$ sees that $M_\alpha$ is countable, it has some $M$-generic filter for $\q(\p_\alpha)^{\lt\omega}$. But since $G_\alpha$ was chosen to be the $L$-least such filter, then by suitability, $M_\beta$ must contain $G_\alpha$. So $M_\alpha[G_\alpha]\subseteq M_\beta$.
\end{proof}
\section{The Kanovei-Lyubetsky Theorem}\label{sec:KanoveiLyubetsky}
We will reprove here the Kanovei-Lyubetsky theorem from \cite{kanovei:productOfJensenReals} showing that Jensen's poset $\p^J$ from Section~\ref{sec:PerfectPosets}, has the property that in a forcing extension $L[G]$ of $L$ by the finite support $\omega$-length product $\p^{J\lt\omega}$, the only $L$-generic reals for $\p$ are the $\omega$-many slices of $G$. This is a generalization of Jensen's uniqueness of generic filters property to products.

Suppose that $\p$ is a perfect poset and $H$ is a generic filter for the product $\p^{\lt\omega}$. We will call $x_n$ the real on the $n$-th coordinate of $H$ and let $\dot x_n$ be its canonical name.

For the next lemma, we suppose that $\p$ is a countable perfect poset that is an element of a suitable model $M$. We should think of $\p$ as one of the perfect posets $\p_\alpha$ arising at stage $\alpha$ in the construction of Jensen's poset $\p^J$ and of $M$ as the model $M_\alpha$ from that stage.
\begin{theorem}[Kanovei-Lyubetsky, \cite{kanovei:productOfJensenReals}]\label{th:denseToAvoidBranchesInProduct}
In $M$, suppose that $\dot r$ is a $\p^{\lt\omega}$-name for a real such that for all $n\in\omega$, $\one\forces_{\p^{\lt\omega}} \dot r\neq \dot x_n$. Then in a forcing extension $M[G]$ by $\q(\p)^{\lt\omega}$, for every generic perfect tree $\T_n$, conditions forcing that $\dot r\notin [\T_n]$ are dense in $\p^{*\lt\omega}$.
\end{theorem}
\begin{proof}
Fix a condition $p\in \p^{*\lt\omega}$ and $d\in\omega$. We need to find $p'\leq p$ such that $p'\forces \dot r\not\in [\T_d]$. Since $\U^{\lt\omega}$ is dense in $\p^{*\lt\omega}$, we can assume without loss that $p\in \U^{\lt\omega}$. So let $$p=\la \T_{i_0}\wedge S_0,\T_{i_1}\wedge S_1,\ldots,\T_{i_m}\wedge S_m\ra.$$ By strengthening if necessary, we can assume that for some $n\leq m$, $i_n=d$. As we argued in the proof of Proposition~\ref{prop:propertiesOfProductP*}, there is a condition $q\in G$ with $q(n)=(T_n,k_n)$ such that for every $n\leq m$, there is a node $s_n$ on level $k_{i_n}$ of $T_{i_n}$ such that $(T_{i_n})_{s_n}\leq S_n$ and for $n\neq n'$, $s_n\neq s_{n'}$.

Now we are going to construct a condition $a_q=\la W_0,\ldots,W_{l'},\ldots,W_l\ra\in\p^{\lt\omega}$ associated to $q$ and satisfying the following properties.
\begin{enumerate}
\item For every $n\leq m\leq l'$, $W_n\leq (T_{i_n})_{s_n}$.
\item For every node $s$ on level $k_d$ of $T_d$, there is $i\leq l'$ such that $W_i\leq (T_d)_s$.
\item For every $n\leq l'$, $a_q\forces_{\p^{\lt\omega}}\dot r\notin[W_n]$.
\end{enumerate}
Let $a=\la (T_{i_0})_{s_0},(T_{i_1})_{s_1},\ldots,(T_{i_m})_{s_m}, (T_d)_{r_0},\ldots,(T_d)_{r_k}\ra$, where $r_0,\ldots,r_k$ are nodes on level $k_d$ of $T_d$ excluding those amongst the $s_n$. Since $\one_{\p^{\lt\omega}}\forces \dot r\neq \dot x_n$ for all $n\in\omega$, we can strengthen $a$ to a condition $a_q=\la W_0,W_1,\ldots,W_{m+k+2},\ldots,W_l\ra$ forcing that $\dot r\notin [W_n]$ for all $n\leq m+k+2$.

Using $a_q$, we construct the following condition $\bar q\leq q$. For $i\in\{i_n\mid n\leq m\}$, let $R_i$ be the tree we get by replacing $(T_i)_{s_n}$ with $W_n$ in $T_i$ whenever $i_n=i$ and if $i=d$, then we also replace all $T_{r_n}$ with the appropriate $W_j$ as well. Let $\bar q(i)=(R_i, k_i)$ for $i\in \{i_n\mid n\leq m\}$ and $\bar q(i)=q(i)$ otherwise. By density, some such condition $\bar q$, constructed from $a_q$, is in $G$. It follows that each $(\T_{i_n})_{s_n}\leq W_n\wedge S_n$ and $\T_d$ has a level $k_d$ such that for every node $s$ on level $k_d$, $\T_d\leq W_j$ for some $j$. The first part gives us that $p=\la \T_{i_0}\wedge S_0,\T_{i_1}\wedge S_1,\ldots,\T_{i_m}\wedge S_m\ra$ is compatible with $a_q$. Finally, we would like to argue that $a_q\forces_{\p^{*\lt\omega}} \dot r\notin[\T_d]$.

Let $H^*$ be any $V$-generic filter for $\p^{*\lt\omega}$ containing $a_q$ ($V=M[G]$). Since every maximal antichain of $\p^{\lt\omega}$ in $M$ remains maximal in $\p^{*\lt\omega}$, it follows that $H^*$ restricts to an $M$-generic filter $H$ for $\p^{\lt\omega}$ and $a_q\in H$. Thus in $M[H]$, $\dot r_H$ is not a branch through any $W_n$. But this is absolute, and so $\dot r_H$ is not a branch through $\T_d$ in $V[H^*]$.
\end{proof}

\begin{theorem}[Kanovei-Lyubetsky, \cite{kanovei:productOfJensenReals}]\label{th:uniquenessOfGenericsProduct}
Suppose $H\subseteq \p^{J\lt\omega}$ is $L$-generic. If $r\in L[H]$ is $L$-generic for $\p^J$, then $r=x_n$ for some $n<\omega$.
\end{theorem}
\begin{proof}
Let's suppose that $r\in L[H]$ is a real which is not one of the $x_n$. Let $\dot r$ be a nice $\p^{J\lt\omega}$-name for $r$ such that for all $n\in\omega$, $\one_{\p^{j\lt\omega}}\forces \dot r\neq \dot x_n$.

Choose some transitive model $M\prec L_{\omega_2}$ of size $\omega_1$ with $\dot r\in N$. We can decompose $M$ as the union of a continuous elementary chain of countable substructures $$X_0\prec X_1\prec\cdots\prec X_\alpha\prec\cdots\prec M$$ with $\dot r\in X_0$. By properties of $\diamondsuit$, there is some $\alpha$ such that $\alpha=\omega_1\cap X_\alpha$, $\p_\alpha=\p^J\cap X_\alpha$, and $S_\alpha$ codes $X_\alpha$. Let $M_\alpha$ be the collapse of $X_\alpha$. Then $\p_\alpha$ is the image of $\p^J$ under the collapse and $\alpha$ is the collapse of $\omega_1$. Clearly $\dot r$ is fixed by the collapse because it is a nice name and all antichains of $\p^{J\lt\omega}$ are countable (by Theorem~\ref{th:JensenForcingCCC}). So at stage $\alpha$ in the construction of $\p^J$, we chose a forcing extension $M_\alpha[G]$ by $\q(\p_\alpha)^{\lt\omega}$ and let $\p_{\alpha+1}=\p_\alpha^*$ as constructed in $M_\alpha[G]$. By elementarity, $M_\alpha$ satisfies that $\one_{\p_\alpha^{\lt\omega}}\forces\dot r\neq \dot x_n$ for all $n\in\omega$. Thus, by Theorem~\ref{th:denseToAvoidBranchesInProduct}, for every $n<\omega$, $\p_{\alpha+1}^{\lt\omega}$ has a maximal antichain $\mathcal A_n\in M_\alpha[G]$ consisting of conditions $q$ such that $q\forces_{\p_{\alpha+1}^{\lt\omega}}\dot r\notin [\T_n]$. It follows, using Proposition~\ref{prop:MaxAntiMForcingExtension}, that all the antichains $\mathcal A_n$ remain maximal in $\p^{J\lt\omega}$. Also, by Proposition~\ref{prop:MaxAntiMForcingExtension}, $\la \T_n\mid n<\omega\ra$ remains maximal in $\p^J$.

So let's argue that if $q\in \mathcal A_n$, then $q\forces_{\p^{J\lt\omega}}\dot r\notin [\T_n]$. Let $\bar H^*\subseteq \p^{J\lt\omega}$ be an $L$-generic filter containing $q$ and let $\bar H$ be the restriction of $\bar H^*$ to $\p_{\alpha+1}^{\lt\omega}$. Since $q\in \bar H$, it follows that $\dot r_{\bar H}\notin [\T_n]$ holds in $M_\alpha[G][\bar H]$, but this statement is absolute and so also holds in $L[\bar H^*]$. Since $H$ must meet every $\mathcal A_n$, it holds in $L[H]$ that $\dot r_H=r$ is not a branch through any $\T_n$. So $r$ cannot be $L$-generic for $\p^J$.
\end{proof}
\section{Finite iterations of perfect posets}
An iteration of perfect posets is an iteration of forcing notions in which every initial segment forces that the next poset is perfect. Here we will only be dealing with finite iterations of perfect posets, so we are not concerned with issues of support.
\begin{definition} A \emph{finite iteration of perfect posets} is a finite iteration $$\p_n=\q_0*\dot \q_1*\cdots*\dot \q_{n-1}$$ such that $\q_0$ is a perfect poset and for $1\leq i<n$,
    \begin{center}
    $\one_{\p_i}\forces``\dot \q_i$ is a perfect poset".\end{center}
\end{definition}
Suppose $G\subseteq \p_n$ is $V$-generic. For $1\leq i<n$, let $G_i$ be the restriction of $G$ to $\p_i$. Let $G(0)=G_1$ and for $1\leq i<n$, let $G(i)=\{p(i)_{G_i}\mid p\in G\}$. Let $r_i$ be the unique real determined by $G(i)$. It is not difficult to see that the sequence of reals $\vec r=\la r_0,\ldots, r_{n-1}\ra$ determines $G$. Elements of $G_1$ are trees with $r_0$ as a branch, and inductively, elements of $G_{i+1}$ are conditions $p\in \p_i$ such that $p\restrict i\in G_i$ and $r_i\in p(i)_{G_i}$.

The analogue of the fusion poset $\q(\p)$ for a finite iteration $\p_n$ of perfect posets is the fusion poset $\q(\p_n)$ whose conditions are pairs $(p,F)$ with $p\in\p_n$ and $F:n\to\omega$, ordered so that $(p_2,F_2)\leq (p_1,F_1)$ whenever $p_2\leq p_1$ and for every $i<n$, we have $F_2(i)\geq F_1(i)$ and $$p_2\restrict i\forces p_1(i)\cap \ex{F_1(i)}2=p_2(i)\cap \ex{F_1(i)}2.$$

Fusion arguments with names for perfect trees require that we have some information about a fixed level $n$ of the tree. We will now argue that there are densely many conditions in $\q(\p_n)$ where this is the case.

Suppose $p\in \p_n$ and $\sigma:n\to \ex{\lt\omega}2$. Following \cite{abraham:jensenRealsIterations}, let's define, by induction on $n$, what it means for $\sigma$ to \emph{lie} on $p$.\footnote{Abraham \cite{abraham:jensenRealsIterations} uses the terminology $\sigma$ is \emph{consistent} with $p$.} For $n=1$, we shall say that $\sigma$ \emph{lies} on $p$ whenever $\sigma(0)\in p(0)$. If $\sigma$ lies on $p$, we shall denote by $p\mid\sigma$ the condition $p(0)_{\sigma(0)}$. Note that $p\mid\sigma\leq p$. So suppose that we have defined when $\sigma$ lies on $p$ for $p\in \p_n$, and for $\sigma$ which lies on $p$, we have defined $p\mid\sigma$ so that $p\mid\sigma\leq p$. Let $p\in \p_{n+1}$. We define that $\sigma$ lies on $p$ if $\sigma\restrict n$ lies on $p\restrict n$ and $(p\restrict n)\mid (\sigma\restrict n)\forces \sigma(n)\in p(n)$. If $\sigma$ lies $p$, we shall denote by $p\mid\sigma$ the condition $\bar p$ such that $\bar p\restrict n=(p\restrict n)\mid (\sigma\restrict n)$ and $\bar p(n)=\dot T$, where $\dot T$ is a $\p_n$-name that is interpreted as $p(n)_{\sigma(n)}$ by any $\p_n$-generic filter containing $(p\restrict n)\mid (\sigma\restrict n)$ and as $p(n)$ otherwise. Clearly this gives that $p\mid\sigma\leq p$.

\begin{definition}
Let $F:n\to\omega$ and $\sigma:n\to\ex{\lt\omega}2$. We shall say that $\sigma$ \emph{lies on levels} $F$ if $\sigma(i)\in \ex{F(i)}2$ for all $i<n$.\footnote{Abraham~\cite{abraham:jensenRealsIterations} uses the terminology $\sigma$ is \emph{bounded} by $F$.} We shall say that a pair $(p,F)$ with $p\in\p_n$ is \emph{determined} if for every $\sigma$ lying on levels $F$, either $\sigma$ lies on $p$ or there is some $i<n$ such that $\sigma\restrict i$ lies on $p\restrict i$ and $(p\restrict i)\mid (\sigma\restrict i)\forces \sigma(i)\notin p(i)$. If $\sigma$ lies on levels $F$ and lies on $p$, we shall say that $\sigma$ \emph{lies} on $(p,F)$.
\end{definition}
\noindent Observe that whenever a pair $(p,F)$ is determined, for every $\sigma$ which lies on $p$, $(p\mid\sigma)\restrict i$ decides the $F(i)$-th level of $p(i)$, and indeed $(p\mid\sigma)\restrict i$ forces that $p(i)$ is the union of the $p\mid\sigma'(i)$ for $\sigma'\restrict i=\sigma\restrict i$ that lie on $(p,F)$.
\begin{proposition} Suppose $\p_n$ is a finite iteration of perfect posets.
\begin{enumerate}
\item Determined conditions $(p,F)$ are dense in $\q(\p_n)$.
\item Given a determined condition $(p,F)$, the finite set $$\{p\mid \sigma\mid \sigma\text{ lies on }(p,F)\}$$ is a maximal antichain below $p$ in $\p_n$.
\end{enumerate}
\end{proposition}
\begin{proposition}\label{prop:InsideDenseOpenSet}
Suppose $\p_n$ is a finite iteration of perfect posets and $D\subseteq \p_n$ is open dense. For any determined $(p,F)\in\q(\p_n)$, there is a determined $(q,F)\leq (p,F)$ satisfying $$q\mid\sigma\in D\text{ whenever }\sigma\text{ lies on }(q,F).$$
\end{proposition}
\noindent For details of proofs, see \cite{abraham:jensenRealsIterations}.

Since we will be working mainly with determined conditions, we will now introduce a kind of normal form for them.
\begin{definition}\label{def:sigmaAssigment}
Suppose $\p_n$ is a finite iteration of perfect posets and $\sigma:n\to \ex{\lt\omega}2$. Let's call a condition $p\in \p_n$ a $\sigma$-\emph{condition} if
\begin{enumerate}
\item $p(0)\leq (\ex{\lt\omega}2)_{\sigma(0)}$,
\item for all $1\leq i<n$, $p\restrict i\forces p(i)\leq (\ex{\lt\omega}2)_{\sigma(i)}$.
\end{enumerate}

Suppose that $X_F$ is a collection of $\sigma:n\to\ex{\lt\omega}2$ lying on levels $F:n\to\omega$. An $X_F$-\emph{assignment} is a function $\varphi:X_F\to \p_n$ such that each $\varphi(\sigma)$ is a $\sigma$-condition and $\varphi(\sigma)\restrict i=\varphi(\sigma')\restrict i$  whenever $\sigma\restrict i=\sigma'\restrict i$.
\end{definition}

To motivate these definitions, consider a determined condition $(p,F)\in \q(\p_n)$ and let $X_F^p$ be the set of all $\sigma$ that lie on it. In this case, the map $\varphi_p$ defined by $\varphi_p(\sigma)=p\mid\sigma$ for all $\sigma\in X_F^p$ is clearly an $X_F^p$-assignment. Thus, a determined condition $(p,F)$ gives us a natural $X_F$-assignment, and now conversely we would like to argue that any $X_F$-assignment has a naturally associated determined condition $(q,F)$.

Suppose $X_F$ and $\varphi$ are as in Definition~\ref{def:sigmaAssigment}. Observe that given any $\sigma,\sigma'\in X_F$, either $\sigma(0)=\sigma'(0)$ and so $\varphi(\sigma)(0)=\varphi(\sigma')(0)$, or $\sigma(0)=s\neq t=\sigma'(0)$ are two nodes on level $F(0)$ such that $\varphi(\sigma)(0)\leq (\ex{\lt\omega}2)_s$ and $\varphi(\sigma')(0)\leq (\ex{\lt\omega}2)_t$. More generally, if $\sigma\neq\sigma'$, then there is some least $i$ such that $\sigma\restrict i=\sigma'\restrict i$ and there are nodes $\sigma(i)=s\neq t=\sigma'(i)$ on level $F(i)$ such that $\varphi(\sigma)\restrict i\forces \varphi(\sigma)(i)\leq (\ex{\lt\omega}2)_s$ and $\varphi(\sigma')\restrict i\forces \varphi(\sigma')(i)\leq(\ex{\lt\omega}2)_t$. It follows, in particular, that for any $i<n$, the conditions $\varphi(\sigma)\restrict i$ for $\sigma\in X_F$ form an antichain.

\begin{proposition}\label{prop:sigmaConditions}
Suppose $X_F$ is a collection of $\sigma:n\to \ex{\lt\omega}2$ on level some $F:n\to\omega$ and $\varphi$ is an $X_F$-assignment. Then there is a determined condition $(q,F)\in \q(\p_n)$ such that
\begin{enumerate}
\item $\sigma$ lies on $(q,F)$ if and only if $\sigma\in X_F$,
\item for every $\sigma\in X_F$,
$$q\mid\sigma(0)=\varphi(\sigma)(0),$$ and for all $i<n$, $(q\mid\sigma)\restrict i$ and $\varphi(\sigma)\restrict i$ extend each other and force $$q\mid\sigma(i)=\varphi(\sigma)(i).$$
\end{enumerate}
\end{proposition}
\begin{proof}
Let $q(0)$ be the union of the $\varphi(\sigma)(0)$ for $\sigma\in X_F$, which is in the perfect poset $\q_0$. Let $q(1)$ be the (canonical) $\p_1$-name for the tree which is the union of the collection of trees given by the interpretation of the name $$\{( \varphi(\sigma)(1),\varphi(\sigma)(0))\mid \sigma\in X_F\}.$$  Since each $\varphi(\sigma)(0)\forces \varphi(\sigma)(1)\in\dot\q_1$, and $\dot \q_1$ is forced to be a perfect poset, it follows that $q(0)\forces q(1)\in \dot\q_1$. Let's see what $q(1)_G$ looks like in a forcing extension $V[G]$ by $\p_1$. If $\varphi(\sigma)(0)$ and $\varphi(\sigma')(0)$ are in $G$, then $\sigma(0)=\sigma'(0)=s$ for some $s$, and so $\varphi(\sigma)(0)=\varphi(\sigma')(0)=p$ for some $p$. Thus, the interpretation $q(1)_G$ is the tree which is the union of the $\varphi(\sigma)(1)_G$ for $\sigma(0)=s$. Similarly, let $q(i)$ be the $\p_i$-name for the tree which is the union of the collection of trees given by the interpretation of the name $\{(\varphi(\sigma)(i), \varphi(\sigma)\restrict i)\mid \sigma\in X_F\}$. Again, we have $q\restrict i\forces q(i)\in\dot\q_i$. Let's see now what $q(i)_G$ looks like in a forcing extension $V[G]$ by $\p_i$. Since $\varphi(\sigma)\restrict i$ for $\sigma\in X_F$ form an antichain, if $\varphi(\sigma)\restrict i$ and $\varphi(\sigma')\restrict i$ are both in $G$, then $\sigma'\restrict i=\sigma\restrict i=\tau$ for some $\tau$. So the interpretation $q(i)_G$ is the union of the $\varphi(\sigma)(i)_G$ for $\sigma\restrict i=\tau$.

First, we argue that every $\sigma\in X_F$ lies on $(q,F)$ and simultaneously show (2). So fix some $\sigma\in X_F$. By construction $\sigma(0)\in q(0)$ and $q(0)_{\sigma(0)}=\varphi(\sigma)(0)$. So assume inductively that for some $i<n$,
\begin{enumerate}
\item $\sigma\restrict i$ lies on $q\restrict i$,
\item $(q\restrict i)\mid(\sigma\restrict i)\leq\varphi(\sigma)\restrict i$,
\item  $\varphi(\sigma)\restrict i\leq (q\restrict i)\mid(\sigma\restrict i)$.
\end{enumerate}
Suppose that $G\subseteq \p_i$ is a $V$-generic filter containing $(q\restrict i)\mid(\sigma\restrict i)$. Then also, by assumption, $\varphi(\sigma)\restrict i\in G$. By definition of $q$, we have $$(q(i)_G)_{\sigma(i)}=\varphi(\sigma)(i)_G.$$ So $\sigma\restrict i+1$ lies on $q\restrict i+1$ and we have
$$(q\restrict i+1)\mid(\sigma\restrict i+1)\leq \varphi(\sigma)\restrict i+1\text{ and }\varphi(\sigma)\restrict i+1\leq (q\restrict i+1)\mid(\sigma\restrict i+1).$$

Now suppose that $\tau:n\to\ex{\lt\omega}2$ lies on $(q,F)$. By definition of $q$, it follows that $\tau(0)=\sigma(0)$ for some $\sigma\in X_F$. So suppose inductively that for some $i<n$, there is $\sigma\in X_F$ such that $\tau\restrict i=\sigma\restrict i$. Since $\tau$ lies on $q$, it follows that $(q\restrict i)\mid (\sigma\restrict i)\forces \tau(i)\in q(i)$. But then there is some $\sigma'$ such that $\sigma\restrict i=\sigma'\restrict i$ and $\tau(i)=\sigma'(i)$. So in the last step, we will obtain $\sigma\in X_F$ such that $\sigma=\tau$.
\end{proof}

We shall call the condition $q$ constructed in Proposition~\ref{prop:sigmaConditions} the \emph{amalgamation} of $\varphi$.
\begin{proposition}
Suppose $(p,F)$ is determined and $q$ is the amalgamation of the $X_F^p$-assignment $\varphi_p$. Then $q\leq p$ and $p\leq q$.
\end{proposition}
\begin{proof}
Clearly $p(0)=q(0)$. So let's suppose inductively that for some $i<n$, we have $q\restrict i\leq p\restrict i$ and $p\restrict i\leq q\restrict i$. We will argue that both $p\restrict i$ and $q\restrict i$ force $p(i)=q(i)$. So suppose $G\subseteq \p_i$ is a $V$-generic filter containing $p\restrict i$, and hence also $q\restrict i$. Since the conditions $(p\mid\sigma)\restrict i$ for $\sigma\in X_F^p$ form a maximal antichain below $p\restrict i$, it follows that one such $(p\mid\sigma)\restrict i$ is in $G$. Thus, $p(i)_G$ is the union of the $(p\mid\sigma'(i))_G$ for $\sigma'\in X_F^p$ with $\sigma'\restrict i=\sigma\restrict i$. Similarly, there is some $\tau\in X_F^p$ such that $(q\mid\tau)\restrict i\in G$, so we must have $\tau\restrict i=\sigma\restrict i$.  Also, $q(i)_G$ is the union of the $(q\mid\sigma'(i))_G$ for $\sigma'\in X_F^p$ with $\sigma'\restrict i=\sigma\restrict i$. Now, by Proposition~\ref{prop:sigmaConditions}, $(q\mid\sigma)\restrict i\forces q\mid\sigma'(i)=p\mid\sigma'(i)$ for all $\sigma'\in X_F^p$ with $\sigma'\restrict i=\sigma\restrict i$. So $p(i)_G=q(i)_G$.
\end{proof}
Let us say that a map $\tau:n\to\ex{\lt\omega}2$ \emph{extends} another map $\sigma:n\to\ex{\lt\omega}2$ if $\tau(i)$ extends $\sigma(i)$ for every $i<n$. Now we would like to argue that given a determined condition $(q,F)$ and strengthenings $q_\sigma\leq q\mid\sigma$ for every $\sigma$ lying on $(q,F)$, we can obtain a determined condition $(\bar q,\bar F)\leq (q,F)$ such that for every $\sigma$ lying on $(q,F)$, there is some $\tau$ lying on $(\bar q,\bar F)$ extending $\sigma$ with $\bar q\mid \tau\leq q_\sigma$. Indeed, we will get the following stronger statement.
\begin{proposition}\label{prop:strengthenDependentAssigment}
Suppose $(q,F)\in \q(\p_n)$ is determined and for every $\sigma$ which lies on $(q,F)$, there is a finite set $\mathscr X_\sigma$ of conditions $p\leq q\mid\sigma$. Then there is a condition $(\bar q,\bar F)\leq (q,F)$, also determined, such that for every $\sigma$ which lies on $(q,F)$ and $p\in \mathscr X_\sigma$, there is $\tau$ which lies on $(\bar q,\bar F)$ and extends $\sigma$ having $\bar q\mid\tau\leq p$. Moreover, for every $\tau$ which lies on $(\bar q,\bar F)$, $\bar q\mid\tau\leq p$ for some $p\in \mathscr X_\sigma$ with $\tau$ extending $\sigma$.
\end{proposition}
\begin{proof}
First, suppose $n=1$. Let $(T,n)$ be a condition in $\q(\p)$ and suppose that for every node $t$ on level $n$ of $T$, we have a finite set $\mathscr X_t$ of trees $S\leq T_t$. Let $T'$ be the tree we get by replacing $T_t$ with the union of $S$ in $\mathscr X_t$ for every $t$ on level $n$ of $T$. Fix $t$ on level $n$ of $T'$. Find a large enough level $n_t\geq n$ such that for every $s$ on level $n _t$ of $T'$, we can thin out $T'_s$ to $P^s$ with the property that $P^s\leq S$ for some $S\in \mathscr X_t$, and moreover for every $S\in \mathscr X_t$, there is some node $s$ such that $P^s\leq S$. Let $\bar T$ be the tree we obtain by replacing $T'_s$ with $P^s$ in $T'$. Let $\bar n\geq n_t$ for every $t$ on level $n$ of $T'$. Clearly the condition $(\bar T,\bar n)$ has all the desired properties.

Next, let's consider the case $n=2$. Let $(q,F)$ be a condition in $\q(\p_2)$ and suppose that for every $\sigma$ which lies on $(q,F)$, we have a finite set $\mathscr X_\sigma$ of conditions $p\leq q\mid\sigma$. For every node $t$ on level $F(0)$ of $q(0)$, let $\mathscr X_t$ be the set of all trees $p(0)$ with $p\in \mathscr X_\sigma$ for $\sigma(0)=t$. Using the case $n=1$, let $(\bar T,\bar n)$ be the condition for the sets $\mathscr X_t$. Fix a node $s$ on level $\bar n$ of $\bar T$ and let $t_s$ be the node on level $F(0)$ of $\bar T$ which $s$ extends. Since $(q,F)$ was determined, $\bar T_s$ decides the $F(1)$-th level of $q(1)$. So let $\dot T^s$ be a $\p_1$-name for the tree we get by replacing, for every node $t$ on level $F(1)$ of $q(1)$, each $q(1)_t$  with the union of $p(1)$ for $p\in \mathscr X_\sigma$ with $\sigma=\{(0,t_s),(1,t)\}$ and $\bar T_s\leq p(0)$. Strengthen $\bar T_s$ to $P^s$ deciding a level $m_s\geq F(1)$ such that for every node $u$ on level $m_s$ of $\dot T^s$ above a node $t$ on level $F(1)$, we can thin out $\dot T^s_u$ to $\dot P^{s,u}$ such that $\dot P^{s,u}\leq p(1)$ for some $p\in \mathscr X_\sigma$ with $\sigma=\{(0,t_s),(1,t)\}$ and $\bar T_s\leq p(0)$, and moreover for every such $p(1)$, there is some node $u$ such that $\dot P^{s,u}\leq p(1)$. We can assume, with some more thinning out, that all $m_s=m$ for a fixed $m$. Let $\bar T'$ be the tree we get by replacing each $\bar T_s$ with $P^s$ in $\bar T$. Let $\dot T^{'s}$ be a $\p_1$-name for the tree we get by replacing, for every node $u$ on level $m$ of $\dot T^s$, each $\dot T^s_u$ with $\dot P^{s,u}$. Let $\bar F=\{(0,\bar n),(1,m)\}$. Let $X_{\bar F}$ consist of $\sigma:2\to\ex{\lt\omega}2$ such that $\sigma(0)=s$ on level $\bar n$ of $\bar T'$ and $\sigma(1)=u$ on level $m$ of $\dot T^{'s}$. Let $\varphi$ be the $X_{\bar F}$-assignment such that $\varphi(\sigma)=\{(0,\bar T'_{\sigma(0)}),(1,\dot T^{'\sigma(0)}_{\sigma(1)})\}$. Let $\bar q$ be the amalgamation of $\varphi$. Clearly the condition $(\bar q,\bar F)$ has all the desired properties.

Finally, assuming that the statement holds for $n$, let's argue that it holds for $n+1$ by mimicking how the argument moves from $n=1$ to $n=2$. Let $(q,F)$ be a condition in $\q(\p_{n+1})$ and suppose that for every $\sigma$ which lies on $(q,F)$, we have a finite set $\mathscr X_\sigma$ of conditions $p\leq q\mid\sigma$. For every $\tau$ which lies on $(q\restrict n,F\restrict n)$, let $\mathscr X_\tau$ be the set of all conditions $p\restrict n$ with $p\in \mathscr X_\sigma$ for $\sigma\restrict n=\tau$. Using the inductive hypothesis for $n$, let $(r,J)$ be the condition for the sets $\mathscr X_\tau$. Fix $\rho$ lying on $(r,J)$, and find $\tau_\rho$ which lies on $(q\restrict n,F\restrict n)$ such that $\rho$ extends $\tau_\rho$. Since $(q,F)$ was determined, $r\mid\rho$ decides the $F(n)$-th level of $q(n)$. So let $\dot T^\rho$ be a $\p_n$-name for the tree we get by replacing, for every node $t$ on level $F(n)$ of $q(n)$, each $q(n)_t$ with the union of $p(n)$ for $p\in \mathscr X_\sigma$ with $\sigma=\tau_\rho\union\{(n,t)\}$ and $r\mid\rho\leq p\restrict n$. Strengthen $r\mid \rho$ to $r'_\rho$ deciding a level $m_\rho\geq F(n)$ such that for every node $u$ on level $m_\rho$ of $\dot T^\rho$ above a node $t$ on level $F(n)$, we can thin out $\dot T^\rho_u$ to $\dot P^{\rho,u}$ such that $\dot P^{\rho,u}\leq p(n)$ for some $p\in \mathscr X_\sigma$ with $\sigma=\tau_\rho\union\{(n,t)\}$ and $r\mid\rho\leq p\restrict n$, and moreover for every such $p(n)$, there is some node $u$ with $\dot P^{\rho,u}\leq p(n)$. We can assume, with some more thinning out, that all $m_\rho=m$ for a fixed $m$. Let $(r',J')$ be the condition we get by applying the inductive hypothesis to $(r,J)$ and $\mathscr Y_\rho=\{r'_\rho\}$. Let $\dot T^{'\rho}$ a $\p_n$-name for the tree we get by replacing, for every node $u$ on level $m$ of $\dot T^\rho$, each $\dot T^\rho_u$ with $\dot P^{\rho,u}$.  Let $\bar F=J'\union\{(n,m)\}$. Let $X_{\bar F}$ consist of $\sigma:n+1\to\ex{\lt\omega}2$ such that $\sigma=\tau\union\{( n,u)\}$ for some $\tau$ which lies on $(r',J')$ with $\tau$ extending $\rho$ lying on $(r,J)$ and $u$ on level $m$ of $\dot T^{'\rho}$. Let $\varphi$ be the $X_{\bar F}$-assignment such that $\varphi(\sigma)=r'\mid(\sigma\restrict n)\union\{( n,T^{'\rho}_{\sigma(n)})\}$, where $\sigma\restrict n$ extends $\rho$ lying on $(r,J)$. Let $\bar q$ be the amalgamation of $\varphi$. Clearly the condition $(\bar q,\bar F)$ has all the desired properties.
\end{proof}
\section{Growing finite iterations of perfect posets}\label{sec:GrowingIterations}

In the construction of Jensen's poset $\p^J$, at nontrivial stages $\alpha$, we used the $\omega$-many perfect trees obtained from a partially generic filter for $\q(\p_\alpha)^{\lt\omega}$ to grow the perfect poset $\p_\alpha$ to $\p_{\alpha+1}$. What we would like to do now is to find an appropriate generalization of this construction for growing a finite iteration $\p_n$ of perfect posets using partially generic filters for the associated fusion poset $\q(\p_n)$. More precisely, we would like the following.

Given a finite iteration $\p_n=\q_0*\dot \q_1*\cdots*\dot \q_{n-1}$ of perfect posets, we would like to be able to extend it to a finite iteration $\p_n^*=\q_0^**\dot\q_1^**\cdots*\dot\q_{n-1}^*$ of perfect posets, constructed in a generic extension of some suitable model $M$, with the following properties:
\begin{enumerate}
\item $\q_0\subseteq \q_0^*$,
\item For all $1\leq i<n$, $\one_{\p_i^*}$ forces that $\dot\q_i$ is a perfect poset and $\dot\q_i^*$ extends it.
\item $\p_n\subseteq \p_n^*$.
\item Every maximal antichain $\mathcal A\in M$ of $\p_n$ remains maximal in $\p_n^*$.
\end{enumerate}
\noindent The next theorem from \cite{abraham:jensenRealsIterations} holds the main idea for constructing $\p_n^*$. The set-up for the theorem is left intentionally vague with the details forthcoming in the next section.

Fix a suitable model $M$ with $\p_n\in M$. We carry out the construction of $\p_n^*$ in $n$-steps, at each step $i$, constructing a $\p_i^*$-name $\dot\q_i^*$ for a perfect poset extending $\dot\q_i$. We extend $\q_0$ to $\q_0^*$ as before in a (carefully chosen) forcing extension of $M$ by $\q(\q_0)^{\lt\omega}$. So suppose inductively that we already extended $\p_i$ to $\p_i^*$ satisfying requirements (1)-(4). Thus, in particular, every maximal antichain $\mathcal A\in M$ of $\p_i$ remains maximal in $\p_i^*$, and so every $V$-generic filter $H^*$ for $\p_i^*$ restricts to an $M$-generic filter $H$ for $\p_i$. From this it follows that $\dot \q_i$ is forced to be a perfect poset by $\p_i^*$ because $M[H]$ satisfies that $(\dot \q_i)_H$ is a perfect poset, and this statement is clearly absolute between $M[H]$ and $V[H^*]$. Fix a carefully chosen $M$-generic filter $G\subseteq \q(\p_{i+1})$. Let $\tau(G)$ be a $\p_i^*$-name for a subset of $\q(\dot \q_i)$ such that in any forcing extension $V[H^*]$ by $\p_i^*$ $\tau(G)$ gets interpreted as the collection of all pairs $(p(i)_{H^*},F(i))$ for $(p,F)\in G$. Provided that the poset $\p_i^*$ contains a kind of master condition for $G$, we will be able to conclude that $\tau(G)$ is $M[H]$-generic for $\q((\dot\q_i)_H)$, where $H$ is the restriction of the $V$-generic filter $H^*\subseteq \p_i^*$.
\begin{theorem}\label{th:genericCondition}
Suppose $\bar p\in \p^*_i$ is such that for every $(p,F)\in G$, $\bar p\leq p\restrict i$. Then
\begin{center}
$\bar p\forces ``\tau(G)$ is an $M[\dot H]$-generic filter for $\q(\dot\q_i)$",
\end{center}
 where $\dot H$ is the canonical name for the restriction of the generic filter to $\p_i$.
\end{theorem}
\begin{proof}
Suppose $H^*\subseteq \p^*_i$ is $V$-generic with $\bar p\in H^*$. Let $H$ be the restriction of $H^*$ to $\p_i$. Let $$K=\tau(G)_{H^*}=\{(p(i)_{H^*},F(i))\mid p\in G\}$$ and $\q_i=(\dot\q_i)_{H}$. First, we argue that $K$ is a filter on $\q(\q_i)$. Suppose for $(p,F)\in G$, $(p(i)_{H},F(i))\leq (T,b)$ in $\q(\q_i)$, so that $b\leq F(i)$ and $p(i)_{H}\cap\ex{b}2=T\cap\ex{b}2$. It follows that there is a $\p_i$-name $\dot T$ for $T$ such that
\begin{center}
$\one_{\p_i}\forces ``p(i)\leq \dot T\text{ and } p(i)\cap\ex{b}2=\dot T\cap\ex{b}2"$.
\end{center}
Let $p'=p\restrict i\union \{(i,\dot T)\}$ and $F'=F\union\{(i,b)\}$. Clearly $(p,F)\leq (p',F')$, which means that $(p',F')\in G$. It follows that $p'(i)_{H}=T$, and so $(T,b)\in K$.

Next, we fix $(p,F)$ and $(p',F')$ both in $G$ and argue that $(p(i)_{H},F(i))$ and $(p'(i)_{H},F'(i))$ are compatible in $K$. There is $(q,J)\in G$ below both $(p,F)$ and $(p',F')$. It follows that $J(i)\geq F(i),F'(i)$, $q\restrict i\forces_{\p_i} q(i)\leq p(i),p'(i)$,
\begin{center}
$q\restrict i\forces_{\p_i} ``q(i)\cap \ex{F(i)}2=p(i)\cap\ex{F(i)}2\text{ and }q(i)\cap \ex{F'(i)}2=p'(i)\cap\ex{F'(i)}2"$,
\end{center}
and $(q(i)_{H},J(i))\in K$. Since $\bar p\leq q\restrict i$ and $\bar p\in H^*$, we have $q\restrict i\in H^*$. Now observe that $q\restrict i$ must force the statements mentioned above over $\p^*_i$ as well because any $V$-generic filter for $\p_i^*$ restricts to an $M$-generic filter for $\p_i$ and the statements in question are absolute.
Thus, $$q(i)_{H}\leq p(i)_{H},p'(i)_{H}$$ and $$q(i)_{H}\cap \ex{F(i)}2=p(i)_{H}\cap\ex{F(i)}2\text{ and }q(i)_{H}\cap \ex{F'(i)}2=p'(i)_{H}\cap\ex{F'(i)}2.$$ So $(q(i)_{H},J(i))\leq (p(i)_{H},F(i)),\,(p'(i)_{H},F'(i))$.

Finally, we have to see that $K$ is $M[H]$-generic. So suppose $D\in M[H]$ is dense open in $\q(\q_i)$. Let $\dot D\in M$ be a $\p_i$-name for $D$ such that
\begin{center}
$\one_{\p_i}\forces_{\p_i} ``\dot D$ is dense open in $\q(\dot \q_i)"$.
\end{center}
In $M$, define $$E=\{(p,F)\in \q(\p_{i+1})\mid p\restrict i\forces (p(i),F(i))\in \dot D\}.$$
We claim that $E$ is dense open in $\q(\p_{i+1})$. It is easy to see that $E$ is open, so let's argue that it is dense.
Fix some $(q,J)\in \q(\p_{i+1})$ and assume without loss that $(q,J)$ is determined. There must be some pair $(\dot T,\dot k)$ such that $\dot T$ is a $\p_i$-name for an element of $\dot\q_i$, $\dot k$ is a $\p_i$-name for a natural number, and $$q\restrict i\forces ``(\dot T,\dot k)\in\dot D\text{ and }(\dot T,\dot k)\leq (q(i),J(i))".$$ The set of conditions which decide the value of $\dot k$ is dense open below $q\restrict i$ in $\p_i$. So, by Proposition~\ref{prop:InsideDenseOpenSet}, there a determined condition $(p',F')\leq (q\restrict i,J\restrict i)$ in $\q(\p_i)$ such that for every $\sigma$ which lies on it, $p'\mid \sigma$ decides that $\dot k=k(\sigma)$. Let $k\in\omega$ be above all the $k(\sigma)$. Then $$p'\forces ``(\dot T, k)\leq (\dot T,\dot k)\leq (q(i),J(i))\text{ and }(\dot T,k)\in \dot D".$$ Let $p=p'\union\{(i,\dot T)\}$ and $F=F'\union\{(i,k)\}$. Clearly $(p,F)\in E$. Let's argue that $(p,F)\leq (q,J)$. By construction $p\leq q$ and each $k(\sigma)\geq J(i)$, so $k\geq J(i)$. Finally, $p'=p\restrict i\forces (\dot T,k)\leq (q(i),J(i))$ and so $p\restrict i\forces \dot T\cap \ex{J(i)}2=q(i)\cap \ex{J(i)}2$.

Fix some $(p,F)\in E\cap G$. Since $\bar p\leq p\restrict i$, it follows that $p\restrict i\in H^*$. Thus, $(p(i)_{H},F(i))\in D\cap K$, completing the argument that $K$ is an $M[H]$-generic filter for $\q(\q_i)$.
\end{proof}
Next, we are going to obtain a stronger version of Theorem~\ref{th:genericCondition} that tells us how to get an $M[H]$-generic filter for $\q(\q_i)^{\lt\omega}$, which is really what we need to extend $\dot \q_i$ to $\dot \q_i^*$. For this, we will need to enlarge our fusion poset $\q(\p_{i+1})$. Let $\bar\q_i(\p_{i+1})$ be the following modification of $\q(\p_{i+1})$. Conditions in $\bar\q(\p_{i+1})$ are pairs $(p,F)$ such that $(p\restrict i,F\restrict i)\in \q(\p_i)$, $p(i)$ is some finite tuple $\la \dot T_0,\dot T_1,\ldots,\dot T_{k-1}\ra$ with $p\restrict i\forces \dot T_j\in \dot\q_i$ for all $j<k$, and $F(i)=f:k\to\omega$. The ordering is $(p',F')\leq (p,F)$ whenever
\begin{enumerate}
\item $(p'\restrict i,F'\restrict i)\leq (p\restrict i, F\restrict i)$, and \\
for $j<k$,
\item $F'(i)(j)\geq F(i)(j)$,
\item $p'\restrict i\forces p'(i)(j)\cap \ex{F(i)(j)}2=p(i)(j)\cap \ex{F(i)(j)}2$.
\end{enumerate}
The point is that if $H^*\subseteq \p_i^*$ is $V$-generic, then $$\la ((\dot T_0)_{H^*},f(0)),\ldots,((\dot T_{k-1})_{H^*},f(k-1))\ra$$ is a condition in $\q((\dot \q_i)_H)^{\lt\omega}$. Suppose now that $G\subseteq \bar \q(\p_{i+1})$ is $M$-generic. Let $\tau(G)$ be a $\p_i^*$-name for a subset of $\q(\q_i)^{\lt\omega}$  such that in any forcing extension $V[H^*]$ by $\p_i^*$, $\tau(G)$ gets interpreted as the collection of all $$\la ((\dot T_0)_{H^*},f(0)),\ldots,((\dot T_{k-1})_{H^*},f(k-1))\ra$$ for $(p,F)\in G$. If it so happens that $\p_i^*$ has a master condition $\bar p$ for $G$, then $\tau(G)$ will be $M[H]$-generic for $\q((\q_i)_H)^{\lt\omega}$, where $H$ is the restriction of the $V$-generic filter $H^*\subseteq \p_i^*$.
\begin{theorem}\label{th:genericConditionProduct}
Suppose $\bar p\in \p^*_i$ is such that for every $(p,F)\in G$, $\bar p\leq p\restrict i$. Then
\begin{center}
$\bar p\forces ``\tau(G)$ is an $M[\dot H]$-generic filter for $\q(\dot\q_i)^{\lt\omega}$",
\end{center}
where $\dot H$ is the canonical name for the restriction of the generic filter to $\p_i$.
\end{theorem}
\noindent The proof is essentially the same as of Theorem~\ref{th:genericCondition}. Using Theorem~\ref{th:genericConditionProduct}, we can let $\dot\q_i^*$ be the canonical $\p_i^*$-name for the extension of $\dot \q_i$ formed in $M[\dot H][\tau(G)]$, where $\dot H$ is the restriction of the generic filter to $\p_i$. In the next section, we will show how to obtain the required $M$-generic filters $G$ so that the inductive assumptions hold for $\p_i^*$.
\section{Tree iterations of perfect posets}\label{sec:treeIterations}
Let's define that an \emph{$\omega$-iteration of perfect posets} is a sequence $$\vec P=\la \p_n\mid n<\omega\ra,$$ where $\p_0=\{\emptyset\}$ is a trivial poset and each $\p_n$, for $n\geq 1$, is a finite iteration of perfect posets with the coherence requirement that for $0<m<n$, $\p_n\restrict m=\p_m$. The initial poset $\p_0$ is included in the sequence to make the subsequent definitions more uniform. For this reason, we will also make the ad hoc definition that $\q(\p_0)=\{\emptyset\}$. Note that an $\omega$-iteration of perfect posets is not itself an iteration, rather it is a coherent sequence of finite iterations.

A \emph{tree iteration} is a non-linear forcing iteration along some tree. Given a tree of height $\omega$, the tree iteration of perfect posets will use an $\omega$-iteration $\vec P=\la \p_n\mid n<\omega\ra$ of perfect posets with conditions assigned to level $n$ nodes of the tree coming from the poset $\p_n$. Conditions will be assigned to the nodes coherently so that if a node $s$ on level $n$ extends a node $t$ on level $m$, then the condition $p$ on node $s$ will be such that $p\restrict m$ is on node $t$.
\begin{definition}
Let $\vec P=\la \p_n\mid n<\omega\ra$ be an $\omega$-iteration of perfect posets and let $\mathscr T$ be a tree of height $\omega$.  A \emph{$\mathscr T$-iteration of perfect posets} is the following partial order $\p(\vec P,\mathscr T)$. Conditions in $\p(\vec P,\mathscr T)$ are functions $f_X$ with domain some finite subtree $X$ of $\mathscr T$ such that:
\begin{enumerate}
\item For every node $s$ on level $n$ of $X$, $f_X(s)\in \p_n$.
\item Whenever $s\leq t$ are two nodes in $X$, then $f_X(t)\restrict \lev(s)=f_X(s)$.
\end{enumerate}
The ordering is $f_Y\leq f_X$ whenever $Y$ extends $X$ and for every node $s\in X$, $f_Y(s)\leq f_X(s)$.
\end{definition}
The analogue of the fusion posets $\q(\p)$ and $\q(\p_n)$ for $\p(\vec P,\mathscr T)$ will be the fusion poset $\q(\vec P,\mathscr T)$ whose conditions  are functions $f_X$ with domain some finite subtree $X$ of $\mathscr T$ such that:
\begin{enumerate}
\item For every node $s$ on level $n$ of $X$, $f_X(s)\in \q(\p_n)$.
\item Whenever $s\leq t$ are two nodes in $X$, with $f_X(s)=(p_s,F_s)$ and $f_X(t)=(p_t,F_t)$, then $p_t\restrict \lev(s)=p_s$ and $F_t\restrict\lev(s)=F_s$.
\end{enumerate}
The ordering is $f_Y\leq f_X$ whenever $Y$ extends $X$ and for every node $s\in X$, $f_Y(s)\leq f_X(s)$.
\begin{proposition}\label{prop:strengthenTreeCondition} $\,$
\begin{enumerate}
\item Suppose $f_X\in \p(\vec P, \mathscr T)$, $f_X(s)=p$, and $q\leq p$. Then there is a condition $g_X\leq f_X$ in $\p(\vec P,\mathscr T)$ with $g_X(s)=q$.
\item Suppose $f_X\in \q(\vec P,\mathscr T)$, $f_X(s)=(p,F)$, and $(q,F')\leq (p,F)$. Then there is $g_X\leq p_X$ such that $g_X(s)=(q,F')$.
\end{enumerate}
\end{proposition}
\begin{proof}
We will only prove (1) because the proof of (2) is analogous. Define $g_X$ as follows.  Fix a node $t\in X$. If $t\leq s$, then let $g_X(t)=q\restrict \lev(t)$. If $s\leq t$, then let $g_X(t)$ be $q$ concatenated with the tail of $f_X(t)$.  Otherwise, let $t'$ be the largest node that is compatible with both $t$ and $s$. Let $g_X(t)$ be $q\restrict\lev(t')$ concatenated with the tail of $f_X(t)$.
\end{proof}
Let us call a condition $f_X\in \q(\vec P,\mathscr T)$ \emph{determined} if every $f_X(s)$ is determined. Clearly a condition $f_X$ is determined if and only if conditions on the terminal nodes are determined.

\begin{proposition}\label{prop:determinedInTreeDense}
The set of all determined conditions is dense in $\q(\vec P,\mathscr T)$.
\end{proposition}
\begin{proof}
Fix $f_X\in \q(\vec P,\mathscr T)$. Let $\la s_i\mid i\leq m\ra$ be an enumeration of the terminal nodes of $X$. Using the construction in the proof of Proposition~\ref{prop:strengthenTreeCondition}, strengthen $f_X$ to $f^0_X$ such that $f^0_X(s_0)$ is determined. Inductively, given $f^i_X$, we let $f^{i+1}_X$ be the strengthening of $f^i_X$ constructed as in the proof of Proposition~\ref{prop:strengthenTreeCondition} such that $f^{i+1}_X(s_{i+1})$ is determined. Let's argue that $f^m_X$ is determined. We can assume inductively that the conditions $f^{m-1}_X(s_i)$ for $i<m$ are determined. By construction $f^m_X(s_m)$ is determined. So fix $i<m$. Let $s$ be the node where the branch of $s_i$ and the branch of $s_m$ split. By construction, $f^m_X(s_i)$ is $f^m_X(s_m)\restrict \lev(s)$ concatenated with the tail of $f^{m-1}_X(s_i)$. Clearly $f^m_X(s_m)\restrict \lev(s)$ is determined and $f^m_X(s_m)\restrict \lev(s)\leq f^{m-1}_X(s_i)\restrict \lev(s)$. Now using the definition of what it means to be determined, it is easy to see that the whole condition is determined.
\end{proof}
We will initially consider tree iterations along the countable tree $\ex{\lt\omega}\omega$, and then later extend our results to tree iterations along the uncountable tree $\ex{\lt\omega}\omega_1$.

It is easy to see that the poset $\q(\q_0)^{\lt\omega}$ completely embeds into $\q(\vec P,\ex{\lt\omega}\omega)$ via the map sending a condition to the corresponding tree of height $\leq 2$.
\begin{proposition}\label{prop:completeEmbeddingProductTreeLevel1}
The poset $\q(\q_0)^{\lt\omega}$ completely embeds into $\q(\vec P,\ex{\lt\omega}\omega)$ via the following map $\varphi$:
\begin{enumerate}
\item $\varphi(\one_{\q(\q_0)})=f_X$, where $X$ consists of the root node $s$ and $f_X(s)=\emptyset$.\\
For $p$ with some non-trivial $p(i)$,
\item $\varphi(p)=f_X$, where $X$ consists of the root node together with nodes $\la i\ra$ for non-trivial $p(i)$, such that $f_X(\la i\ra)=p(i)$.
\end{enumerate}
\end{proposition}
\noindent More generally, for each node $s$ on level $n$ of $\ex{\lt\omega}\omega$, $\bar\q(\p_{n+1})$ completely embeds into $\q(\vec P,\ex{\lt\omega}\omega)$ via the map sending a condition to the corresponding tree of height $\leq n+2$ whose stem stretches up to $s$.
\begin{proposition}\label{prop:completeEmbeddingProductTree}
Fix a node $s$ on level $n$ of $\ex{\lt\omega}{\omega}$. The poset $\bar\q(\p_{n+1})$ completely embeds into $\q(\vec P,\ex{\lt\omega}\omega)$ via the following map $\varphi_s$:
\begin{enumerate}
\item $\varphi_s(\one_{\bar \q(\p_{n+1})})=f_X$, where $X$ is the branch ending in $s$, such that $f_X(s)=\one_{\q(\p_n)}$.\\ For $(p,F)$ with $\text{dom}(p(n))=k$,
\item $\varphi_s((p,F))=f_X$, where $X$ consists of the branch ending in $s$ together with nodes $\la s\concat i\ra$ for $i<k$, such that $f_X(s)=(p\restrict n,F\restrict n)$ and
\begin{center}
$f_X(s\concat i)=(p\restrict n\concat p(n)(i),F\restrict n\concat F(n)(i))$.
\end{center}
\end{enumerate}
\end{proposition}
\noindent Suppose $G\subseteq \q(\vec P,\ex{\lt\omega}\omega)$ is $V$-generic and fix some node $s$ on level $n$ of $\ex{\lt\omega}{\omega}$. We will use the notation $G_s$ for the $V$-generic filter for $\bar\q(\p_{n+1})$ added by $G$ via the embedding $\varphi_s$ and we will use the notation $G_\emptyset$ for the $V$-generic filter for $\q(\q_0)^{\lt\omega}$ added by $G$ via the embedding $\varphi$.

Suppose now that $\vec P=\la \p_n\mid n<\omega\ra$ is an $\omega$-iteration of perfect posets that is an element of a suitable model $M$. Let $G\subseteq \q(\vec P,\ex{\lt\omega}\omega)$ be $M$-generic. We will argue that in $M[G]$, we can grow each iteration $\p_n$ to an iteration $\p_n^*$ satisfying requirements (1)-(4) from Section~\ref{sec:GrowingIterations}:
\begin{enumerate}
\item $\q_0\subseteq \q_0^*$,
\item For all $1\leq i<n$, $\one_{\p_i^*}$ forces that $\dot\q_i$ is a perfect poset and $\dot\q_i^*$ extends it.
\item $\p_n\subseteq \p_n^*$.
\item Every maximal antichain $\mathcal A\in M$ of $\p_n$ remains maximal in $\p_n^*$.
\end{enumerate}

It is straightforward to extend $\q_0$ to $\q^*_0$. By Proposition~\ref{prop:completeEmbeddingProductTreeLevel1}, $G$ adds an $M$-generic filter $G_\emptyset$ for $\q(\q_0)^{\lt\omega}$. Let $\T_i^0$ for $i<\omega$ be the generic perfect trees added by $G_\emptyset$ and construct $\q_0^*$ as before. Recall that $\la \T_i^0\mid i<\omega\ra$ is a maximal antichain in $\q_0^*$ and every maximal antichain $\mathcal A\in M$ of $\q_0$ remains maximal in $\q_0^*$.

Now let's show how to extend $\dot\q_1$ to a $\p_1^*$-name $\dot\q^*_1$ for a perfect poset.  By Proposition~\ref{prop:completeEmbeddingProductTree}, for each node $s$ on level 1 of the tree $\ex{\lt\omega}\omega$, $G$ adds an $M$-generic filter $G_s$ for $\bar\q(\p_2)$. Observe that each $\T_i^0\leq p\restrict 1$ for all $p$ with $(p,F)\in G_{\la i\ra}$. Thus, by Theorem~\ref{th:genericConditionProduct}, whenever $H^*\subseteq \p_1^*$ is a $V$-generic filter containing $\T_i^0$, then the interpretation $\tau(G_{\la i\ra})_{H^*}$  is an $M[H]$-generic filter for $\q((\dot\q_1)_{H})^{\lt\omega}$, where $H$ is the restriction of $H^*$ to $\p_1$. Let $\tau$ be a mixed $\p_1^*$-name that is interpreted as $\tau(G_{\la i\ra})_{H^*}$, whenever $\T_i^0\in H^*$. Since the conditions $\T_i^0$ form a maximal antichain in $\p_1^*$,
\begin{center}
$\one_{\p_1^*}\forces``\tau$ is an $M[\dot H]$-generic filter for $\q(\dot\q_1)^{\lt\omega}$,"
\end{center}
where $\dot H$ is the canonical name for the restriction of the generic filter to $\p_1$. So let $\dot \q_1^*$ be a $\p_1^*$-name for the perfect poset constructed as usual from $\dot \q_1$ and $\tau$.

For each $j<\omega$, we can choose a $\p_1^*$-name $\dot \T_j^1$ for the perfect tree on coordinate $j$ of $\tau$. The pairs $(\T_i^0,\dot \T_j^1)$ form a maximal antichain in $\p_2^*$ because the collection of $\T_i^0$ is a maximal antichain and each $\T_i^0$ forces that the trees $\dot T_j^1$ form a maximal antichain in $\dot\q_1^*$. Also, clearly for every $(p,F)\in G_{\la i, j\ra}$, we have $(\T_i^0,\dot \T_j^1)\leq p\restrict 2$. This will allow us to apply Theorem~\ref{th:genericConditionProduct} to grow $\dot \q_2$ to $\dot \q_2^*$. Finally, let's argue that $\p_2$ is actually a subset of $\p_2^*$. Suppose $p$ is a condition in $\p_2$. Then $p(0)\in \p_1$ and hence $p(0)\in \p_1^*$ as well. Also, clearly $p(1)$ is a $\p_1^*$-name. So we need to argue that $p(0)\forces_{\p_1^*}p(1)\in \dot\q_1^*$. Fix a $V$-generic filter $H^*\subseteq \p_1^*$ with $p(0)\in H^*$ and consider $M[H]$ where $H$ is the restriction of $H^*$ to $\p_1$. Since $p(0)\forces_{\p_1} p(1)\in \dot \q_1$, in $M[H]$, we have $p(1)_H\in (\dot \q_1)_H\subseteq (\dot \q_1^*)_{H^*}$. It remains to show (4), that every maximal antichain of $\p_2$ from $M$ remains maximal in $\p_2^*$. We will provide an inductive proof of this later in Lemma~\ref{lem:antichainSuccessorCase}.

For now to finish the construction, we assume that properties (1)-(4) hold for $\p_n^*$. We will additionally assume that:
\begin{enumerate}
\item For each node $s$ on level $n$ of $\ex{\lt\omega}\omega$ there is a condition $$(\T_{s(0)}^0,\dot\T_{s(1)}^1,\ldots,\dot \T_{s(n-1)}^{n-1})\in \p_n^*$$ which is below all $p\restrict n$ with $(p,F)\in G_s$.
\item The collection of all such conditions $(\T_{s(0)}^0,\dot\T_{s(1)}^1,\ldots,\dot \T_{s(n-1)}^{n-1})$ is a maximal antichain in $\p_n^*$.
\end{enumerate}
With this set-up, we extend $\dot \q_n$ to $\dot \q_n^*$ identically to the case $n=1$ above, using Theorem~\ref{th:genericConditionProduct}. It is also easy to see inductively that $\p_n$ is a subset of $\p_n^*$.

To get some intuition for the construction, let's fix an $L$-generic filter $H^*\subseteq\p_1^*$ and see what a condition $\T^1_j=(\dot\T^1_j)_{H^*}$ looks like. Since $\la\T^0_i\mid i<\omega\ra$ is a maximal antichain of $\p_1^*$, there is a unique $i$ such that $\T^0_i\in H^*$. Let $$K=\{(p(1)_{H^*},F(1))\mid f_X\in G\text{ and }f_X(\la i,j\ra)=(p,F)\}.$$ Then $\T_j^1$ is the union of $T\cap \ex{n}2$ for $(T,n)\in K$.

To give the promised argument that every maximal antichain of $\p_n$ from $M$ remains maximal in $\p_n^*$, we first need to define the analogue of $\U\subseteq \p^*$ from Section~\ref{sec:PerfectPosets} for $\p_n^*$.

Let $\U_n$ be the subset of $\p_n^*$ consisting of conditions $p=(p_0,\dot p_1,\ldots,\dot p_{n-1})$ such that $p_0=\T^0_{j_0}\wedge S_0$ for some $j_0<\omega$ and $S_0\in \q_0$, and for $i<n$, $p\restrict i\forces \dot p_i=\dot\T^i_{j_i}\wedge \dot S_i$ for some $j_i<\omega$ and a $\p_i$-name $\dot S_i$ such that $\one_{\p_i}\forces \dot S_i\in\dot\q_i$.

\begin{proposition}\label{prop:UnDenseInP*n}
$\U_n$ is dense in $\p^*_n$.
\end{proposition}
\begin{proof}
We argue by induction on $n$. By Proposition~\ref{prop:propertiesOfP*}, $\U_1$ is dense in $\p_1^*$ and every maximal antichain of $\p_1$ from $M$ remains maximal in $\p_1^*$. So let's suppose inductively that $\U_n$ is dense in $\p_n^*$ and every maximal antichain of $\p_n$ from $M$ remains maximal in $\p_n^*$. This argument is meant to take place simultaneously with the inductive proof of Lemma~\ref{lem:antichainSuccessorCase}, where we use the density of $\U_n$ to argue that maximal antichains stay maximal.

Fix a condition $(q_0',\dot q_1',\ldots, \dot q_n')$ in $\p_{n+1}^*$. In particular, we have $(q_0',\dot q_1',\ldots,\dot q_{n-1}')\forces \dot q_n'\in \q_n^*$. Consider any forcing extension $V[H^*]$ by $\p_n^*$ with $(q_0',\dot q_1',\ldots,\dot q_{n-1}')\in H^*$. In $V[H^*]$, there is $j_n<\omega$ and $S\in (\dot \q_n)_{H^*}$ such that $(\dot \T_{j_n}^n)_{H^*}\wedge S\leq (q_n')_{H^*}$. So $S\in M[H]$, where $H$ is the restriction of $H^*$ to $\p_n$. Let $\dot S$ be a $\p_n$-name in $M$ such that $(\dot S)_H=S$ and let $p\in H$ force that $\dot S\in \dot\q_n$. Let $\mathcal A\in M$ be any maximal antichain of $\p_n$ extending $\{p\}$ and let $\dot S_n$ be the mixed name such that $p\forces \dot S=\dot S_n$ and every other $q\in \mathcal A$ forces that $\dot S_n=\ex{\lt\omega}2$. So $\one_{\p_n}\forces \dot S_n\in \dot\q_n$. Now let $(q_0,\dot q_1,\ldots,\dot q_{n-1})\leq (q_0',\dot q_1',\ldots,\dot q_{n-1}')$ be a condition in $H^*$ forcing that $\dot{\T}^n_{j_n}\wedge \dot S_n\leq \dot q_n'$. We have just argued that there is a condition $(q_0,\dot q_1,\ldots,\dot q_{n-1})\leq (q_0',\dot q_1',\ldots,\dot q_{n-1}')$ and a $\p_n$-name $\dot S_n$ such that $\one_{\p_n}\forces \dot S_n\in \dot\q_n$ and $$(q_0,\dot q_1,\ldots,\dot q_{n-1})\forces \dot{\T}^n_{j_n}\wedge \dot S_n\leq \dot q_n'$$ for some $j_n<\omega$.

Now, by the inductive assumption, there is a condition $$(p_0,\dot p_1,\ldots,\dot p_{n-1})\leq (q_0,\dot q_1,\ldots, \dot q_{n-1})$$ in $\U_n$. So clearly the condition $(p_0,\dot p_1,\ldots,\dot p_{n-1},\dot{\T}^n_{j_n}\wedge \dot S_n)\in\U_{n+1}$.
\end{proof}

\noindent We will usually abuse notation by writing conditions in $\U_n$ in the form $$(\T^0_{j_0}\wedge S_0,\dot \T^1_{j_1}\wedge \dot S_1,\ldots,\dot \T^{n-1}_{j_{n-1}}\wedge \dot S_{n-1}).$$ The next lemma is a generalization of Proposition~\ref{prop:subtreeContainedInS}.
\begin{lemma}\label{lem:subtreeContainedinSiteration}
Suppose $f_X\in \q(\vec P,\ex{\lt\omega}\omega)$ forces that $$(\T^0_{j_0}\wedge S_0,\dot \T_{j_1}^1\wedge \dot S_1,\ldots,\dot \T_{j_{n-1}}^{n-1}\wedge \dot S_{n-1})\in \U_n.$$ Then there is a condition $g_Y\leq f_X$ such that $g_Y(\la j_0,j_1,\ldots,j_{n-1}\ra)=(q,F)$ is determined and $\tau:n\to \ex{\lt\omega}2$ lies on $(q,F)$ such that
\begin{center}
$q(0)_{\tau(0)}\leq S_0$, and for all $1\leq i<n$, $(q\restrict i)\mid(\tau\restrict i)\forces q(i)_{\tau(i)}\leq \dot S_i$.
\end{center}
\end{lemma}
\begin{proof}
The case $n=1$ follows by Proposition~\ref{prop:subtreeContainedInS}. So suppose $n=2$. Recall that every maximal antichain of $\p_1$ from $M$ remains maximal in $\p_1^*$. By strengthening $f_X$, we can assume that $\la j_0,j_1\ra\in X$. Let $f_X(\la j_0,j_1\ra)=(\bar q,\bar F)$. By strengthening further, using the case $n=1$, we can assume that there is a node $s_0$ on level $\bar F(0)$ such that $\bar q(0)_{s_0}\leq S_0$. By strengthening some more, we can assume that $(\bar q,\bar F)$ is determined.

Let's argue that $\bar q(0)_{s_0}$ forces that for some node $s_1$ on level $\bar F(1)$ of $\bar q(1)$, $\bar q(1)_{s_1}\wedge \dot S_1\neq\emptyset$. If this is not the case, then there is some $T\leq\bar q(0)_{s_0}$ which forces that there is no such node. Let $T'$ be the tree we get by replacing $\bar q(0)_{s_0}$ with $T$ in $\bar q(0)$. Let $\bar q'=(T',\bar q(1))$ and let $f'_X$ be a condition strengthening $f_X$ so that $f'_X(\la j_0,j_1\ra)=(\bar q',\bar F)$. Let $G$ be an $M$-generic filter for $\q(\vec P,\ex{\lt\omega}\omega)$ containing $f'_X$ and consider $M[G]$. The tree $R=\T_{j_0}^0\wedge S_0\wedge\bar q'(0)_{s_0}(=T)$ is a condition in $\p_1^*$, and so we can fix some $V$-generic filter $H^*\subseteq \p^*_1$ containing $R$. In particular, $H^*$ contains $\T_{j_0}^0\wedge S_0$ and $T$. Thus, in $M[H]$ (where $H$ is the restriction of $H^*$ to $\p_1$), $\bar q(1)_{H}$ does not have any subtree on level $F(1)$ which has a common subtree with $(\dot S_1)_{H}$. But then $(\dot\T_{j_1}^1)_{H^*}\wedge (\dot S_1)_{H}=\emptyset$ contradicting our assumption that $f_X$ forced $(\T_{j_0}^0\wedge S_0,\dot \T_{j_1}^1\wedge \dot S_1)\in \U_2$.

Let's strengthen $\bar q(0)_{s_0}$ to some $T_0$, so that there is a node $s_1$ on level $\bar F(1)$ such that $T_0\forces \bar q(1)_{s_1}\wedge\dot S_1\neq\emptyset$. Let $\dot T_1$ be a $\p_1$-name such that $$T_0\forces \dot T_1 =\bar q(1)_{s_1}\wedge\dot S_1.$$ Since $(\bar q,\bar F)$ is determined, it follows that $\bar q(0)_{s_0}\forces s_1\in \bar q(1)$. Thus, $\tau=\{(0, s_0), (1,s_1)\}$ lies on $(\bar q,\bar F)$ and $(T_0,\dot T_1)\leq \bar q\mid\tau$. By Proposition~\ref{prop:strengthenDependentAssigment}, there is a determined condition $(q,F)\leq (\bar q,\bar F)$ with $\rho$ lying on $(\bar q,\bar F)$ and extending $\tau$ such that $q\mid\rho\leq (T_0,\dot T_1)$. So finally, we strengthen $f_X$ to a condition $g_X$ with $g_X(\la j_0,j_1\ra)=(q,F)$ using Proposition~\ref{prop:strengthenTreeCondition}~(2).

Suppose inductively that the statement holds for $n$ and every maximal antichain of $\p_n$ from $M$ remains maximal in $\p_n^*$. Let's argue that the statement holds for $n+1$. This will basically be a generalization of the argument passing from $n=1$ to $n=2$. By strengthening $f_X$, we can assume that $\la j_0,\ldots,j_n\ra\in X$ and that $f_X$ is determined. Let $f_X(\la j_0,\ldots,j_n\ra)=(\bar q,\bar F)$.  By using our inductive assumption for $n$ and strengthening further, we can assume that $\tau:n\to \ex{\lt\omega}2$ lies on $(\bar q\restrict n,\bar F\restrict n)$ such that
\begin{center}
$q(0)_{\tau(0)}\leq S_0$, and for all $1\leq i<n$, $(q\restrict i)\mid(\tau\restrict i)\forces q(i)_{\tau(i)}\leq \dot S_i$.
\end{center}
Let's argue that $(\bar q\restrict n)\mid \tau$ forces that for some node $s_n$ on level $\bar F(n)$ of $\bar q(n)$, $\bar q(n)_{s_n}\wedge \dot S_n\neq\emptyset$. If this is not the case, then there is a condition $p\leq (\bar q\restrict n)\mid \tau$ which forces that there is no such node. By Proposition~\ref{prop:strengthenDependentAssigment}, there is a determined condition $(p',F')\leq (\bar q\restrict n,F\restrict n)$ with $\rho$ lying on $(p',F')$ and extending $\tau$ such that $p'\mid\rho\leq p$. Let $\bar q'=p'\union \{(n,\bar q(n))\}$ and $\bar F'=F'\union \{(n,\bar F(n))\}$. Let $f'_X$ be the strengthening of $f_X$ in which $f'_X(\la j_0,\ldots,j_n\ra)=(\bar q',\bar F')$. Let $G$ be an $M$-generic filter for $\q(\vec P,\ex{\lt\omega}\omega)$ containing $f'_X$ and consider $M[G]$.  The condition $$R=(\T_{j_0}^0\wedge S_0\wedge \bar q'(0)_{\rho(0)},\ldots,\dot T_{j_{n-1}}^{n-1}\wedge \dot S_{n-1}\wedge \bar q'(n-1)_{\rho(n-1)})$$ is in $\p^*_n$ (where $(\bar q'(0)_{\rho(0)},\ldots,\bar q'(n-1)_{\rho(n-1)})\leq p$ by construction), and so we can fix a $V$-generic filter $H^*$ containing $R$, and argue as in the case $n=2$ towards a contradiction.

Thus, there is a condition $$(T_0,\dot T_1,\ldots,\dot T_{n-1})\leq (\bar q\restrict n)\mid \tau,$$ a $\p_n$-name $\dot T_n$ for a perfect tree, and a node $s_n$ on level $\bar F(n)$ such that
$$(T_0,\dot T_1,\ldots,\dot T_{n-1})\forces \dot T_n=\bar q(n)_{s_n}\wedge \dot S_n.$$  Let $\sigma=\tau\cup \{( n,s_n)\}$. So $(T_0,\dot T_1,\ldots,\dot T_{n-1},\dot T_n)\leq \bar q\mid \sigma$, and it remains to construct the required condition using Proposition~\ref{prop:strengthenDependentAssigment}.
\end{proof}
We are now ready to prove that every maximal antichain of $\p_n$ from $M$ remains maximal in $\p_n^*$.
\begin{lemma}\label{lem:antichainSuccessorCase}
Every maximal antichain of $\p_n$ from $M$ remains maximal in $\p_n^*$.
\end{lemma}
\begin{proof}
The statement is true for $n=1$. So we can assume inductively that $\U_n$ is dense in $\p_n^*$. Let $\mathcal A\in M$ be a maximal antichain in $\p_n$. It suffices to show that every condition in $\U_n$ is compatible with an element of $\mathcal A$. So let $f_X$ be some condition which forces that $$(\T_{j_0}^0\wedge S_0,\dot T^1_{j_1}\wedge \dot S_1,\ldots, \dot T_{j_{n-1}}^{n-1}\wedge \dot S_{n-1})\in \U_n.$$ By Proposition~\ref{lem:subtreeContainedinSiteration}, we can strengthen $f_X$ to a condition $g_Y$ such that $g_Y(\la j_0,\ldots,j_{n-1}\ra)=(q,F)$ is determined and $\tau:n\to \ex{\lt\omega}2$ lies on $(q,F)$ such that $q(0)_{\tau(0)}\leq S_0$, and for all $1\leq i<n$, $(q\restrict i)\mid (\tau\restrict i)\forces q(i)_{\tau(i)}\leq\dot S_i$.  Since $\mathcal A$ is maximal in $\p_n$, $q\mid\tau$ is compatible with some $p\in \mathcal A$. So let $q'\leq q\mid\tau, p$. By Proposition~\ref{prop:strengthenDependentAssigment}, there is a determined condition $(\bar q,\bar F)\leq (q,F)$ with $\rho$ lying on $(\bar q,\bar F)$ and extending $\tau$ such that $\bar q\mid\rho\leq q'$. Strengthen $g_Y$ to $\bar g_Y$ with $(q,F)$ strengthened to $(\bar q,\bar F)$ as in Proposition~\ref{prop:strengthenTreeCondition}~(2). The condition $\bar g_Y$ forces that $(\T_{j_0}^0\wedge S_0,\dot T^1_{j_1}\wedge \dot S_1,\ldots, \dot T_{j_{n-1}}^{n-1}\wedge \dot S_{n-1})$ is compatible with $p\in\mathcal A$.
\end{proof}
Let $\vec P^*=\la \p_n^*\mid n<\omega\ra$ be the $\omega$-iteration made up of the extended iterations $\p_n^*$.
Let $\U$ be the subset of $\p(\vec P^*,\ex{\lt\omega}\omega)$ consisting of conditions $f_X$ such that for all $s\in X$ on level $n$, $f_X(s)\in \U_n$.
\begin{proposition}\label{prop:UisDenseInP*}
$\U$ is dense in $\p(\vec P^*,\ex{\lt\omega}\omega)$.
\end{proposition}

Next, we will prove the analogue of Proposition~\ref{prop:propertiesOfProductP*}.
\begin{lemma}\label{lem:antichainTreeIteration}
Every maximal antichain of $\p(\vec P,\ex{\lt\omega}\omega)$ from $M$ remains maximal in $\p(\vec P^*,\ex{\lt\omega}\omega)$.
\end{lemma}
\begin{proof}
Fix a maximal antichain $\mathcal A\in M$ of $\p(\vec P,\ex{\lt\omega}\omega)$. By Proposition~\ref{prop:UisDenseInP*}, it suffices to show that every condition $f_X\in \U$ is compatible with some element of $\mathcal A$. So fix $f_X\in \U$. For a node $s\in X$ on level $n$ of $\ex{\lt\omega}\omega$, let $$f_X(s)=(\T_{j_{0,s}}^0\wedge S_{0,s},\ldots,\dot \T_{j_{n-1,s}}^{n-1}\wedge \dot S_{n-1,s}).$$ Let $\la j_{0,s},\ldots,j_{n-1,s}\ra=\vec j_s$, and note that we can have $\vec j_s=\vec j_{s'}$ for $s'\neq s$. Fix a condition $g_Y\in G$ forcing that $f_X\in \U$. Repeatedly using the construction in the proof of Lemma~\ref{lem:subtreeContainedinSiteration} on terminal nodes, we can find a determined condition $\bar g_{Y}\leq g_Y$ with $\bar g_{Y}(\vec j_s)=(\bar q_{\vec j_s},\bar F_{\vec j_s})$ such that for every node $s\in X$, some $\tau_s:n\to \ex{\lt\omega}2$ lies on $(\bar q_{\vec j_s},\bar F_{\vec j_s})$ with
\begin{center}
$\bar q_{\vec j_s}(0)_{\tau_s(0)}\leq S_{0,s}$, and for all $1\leq i<n$, $(\bar q_{\vec j_s}\restrict i)\mid(\tau_s\restrict i)\forces \bar q_{\vec j_s}(i)_{\tau_s(i)}\leq \dot S_{i,s}$.
\end{center}
The construction gives $\tau_s$ satisfying that if $s'$ extends $s$, then $\tau_{s'}$ extends $\tau_s$.

Let $\bar f_X$ be a condition in $\p(\vec P,\ex{\lt\omega}\omega)$ such that for every node $s\in X$, we have $\bar f_X(s)=\bar q_{\vec j_s}\mid \tau_s$. Since $\mathcal A$ is maximal in $\p(\vec P,\ex{\lt\omega}\omega)$, there is $a_Z\in \mathcal A$ compatible with $\bar f_X$. So let $r_W\leq a_Z,\bar f_X$.

We then carry out the construction from the proof of Proposition~\ref{prop:strengthenDependentAssigment}, working our way up the tree instead of up the iteration, to obtain a condition $h_Y\leq \bar g_Y$ with $h_Y(\vec j_s)=(q_{\vec j_s},F_{\vec j_s})$ such that for every $s\in X$, some $\sigma$ lies on $(q_{\vec j_s},F_{\vec j_s})$ having $q_{\vec j_s}\mid \sigma\leq r_W(s)$. The condition $h_Y$ forces that $f_X$ and $a_Z$ are compatible. By density, some such $h_{Y}$ must be in $G$.
\end{proof}
We now summarize our results in the following lemma, which will serve as the analogue of Proposition~\ref{prop:propertiesOfP*}.
\begin{lemma}
Suppose $M$ is a suitable model and $\vec P=\la \p_n\mid n<\omega\ra\in M$ is an $\omega$-iteration of perfect posets. If $G\subseteq \q(\vec P,\ex{\lt\omega}\omega)$ is $M$-generic and $\vec P^*=\la \p_n^*\mid n<\omega\ra$ is constructed in $M[G]$ as above, then
\begin{enumerate}
\item $\vec P^*$ is an $\omega$-iteration of perfect posets,
\item $\q_0\subseteq \q_0^*$ and for all $n<\omega$, $\one_{\p_n^*}$ forces that $\dot \q_n$ is a perfect poset and $\dot \q_n^*$ extends it,
\item $\p_n\subseteq \p_n^*$,
\item every maximal antichain of $\p_n$ from $M$ remains maximal in $\p_n^*$,
\item every maximal antichain of $\p(\vec P,\ex{\lt\omega}\omega)$ from $M$ remains maximal in $\p(\vec P^*,\ex{\lt\omega}\omega)$.
\end{enumerate}
\end{lemma}
Now we will define an $\omega$-iteration $\vec P^J=\la \p_n^J\mid n<\omega\ra$ of perfect posets such that the tree iteration $\p(\vec P^J,\ex{\lt\omega}\omega)$ is going to have the property that in a forcing extension $L[G]$ of $L$ by $\p(\vec P^J,\ex{\lt\omega}\omega)$, the only $L$-generic filters for the iteration $\p_n^J$ are the restrictions of $G$ to a level $n$ node. This will be the generalization of Jensen's uniqueness of generic filters property to tree iterations.

Let $\la S_\alpha\mid \alpha<\omega_1\ra$ be a canonically defined $\diamondsuit$-sequence. We will construct $\vec P^J$ in $\omega_1$-many steps using $\diamondsuit$ to seal maximal antichains along the way.
Let $\vec P_0$ be the $\omega$-iteration of perfect posets where $\q_0=\p_{\text{min}}$ and each $\dot \q_n=\check \p_{\text{min}}$.  Suppose the $\omega$-iteration $\vec P_\alpha=\la \p_n^\alpha\mid n<\omega\ra$ of perfect poset has been defined. We let $\vec P_\alpha=\vec P_{\alpha+1}$, unless the following happens. Suppose $S_\alpha$ codes a well-founded binary relation $E\subseteq \alpha\times\alpha$ such that the collapse of $E$ is a suitable model $M_\alpha$ with $\vec P_\alpha\in M_\alpha$ and $\alpha=\omega_1^{M_\alpha}$. In this case, we take the $L$-least $M_\alpha$-generic filter $G\subseteq \q(\vec P_\alpha,\ex{\lt\omega}\omega)$ and let $\vec P_{\alpha+1}=\vec P_\alpha^*$ as constructed in $M_\alpha[G]$. At limit stages $\lambda$, to obtain the $\omega$-iteration $\vec P_\lambda$, we let $\q^\lambda_0$ be the union of the $\q^\xi_0$ for $\xi<\lambda$, and given that we have defined $\p_n^\lambda$, we let $\dot \q_n^\lambda$ be a $\p^\lambda_n$-name for the poset that is the union of the $\dot \q_n^\xi$ for $\xi<\lambda$.  In order for this limit definition to make sense, we need to verify that each $\dot\q^\xi_n$ is a $\p^\lambda_n$-name for a perfect poset. So let's argue that this is indeed the case.

Clearly $\p_1^\lambda$ makes sense, $\p_1^\xi\subseteq \p_1^\lambda$ for every $\xi<\lambda$, and every maximal antichain  of $\p_1^\xi$ from $M_\xi$ remains maximal in $\p_1^\lambda$. So we can assume inductively that we have defined $\p_n^\lambda$ so that $\p_n^\xi\subseteq\p_n^\lambda$ for every $\xi<\lambda$, and every maximal antichain of $\p_n^\xi$ in $M_\xi$ remains maximal in $\p_n^\lambda$. Let $H^*\subseteq \p_n^\lambda$ be $V$-generic. By our assumption, the filter $H^*$ restricts to an $M_\xi$-generic filter $H$ for $\p^\xi_n$. The model $M_\xi[H]$ satisfies that $(\dot\q^\xi_n)_H$ is a perfect poset and so this must be the case in $V[H^*]$ as well since this is absolute. So we can extend the definition to $\p_{n+1}^\lambda$. It is easy to see that $\p_{n+1}^\xi\subseteq \p_{n+1}^\lambda$. It remains to argue that every maximal antichain $\mathcal A\in M_\xi$ of $\p_{n+1}^\xi$ remains maximal in $\p_{n+1}^\lambda$. Fix $p\in \p_{n+1}^\lambda$. We can assume inductively that we have already showed for every $\mu<\nu<\lambda$, that every maximal antichain of $\p_{n+1}^{\mu}$ in $M_\mu$ remains maximal in $\p_{n+1}^{\nu}$. Then, by definition of $\p^\lambda_{n+1}$, there is some $\xi<\eta<\lambda$ such that $p(0)\in \q_0^{\eta}$, and $p\restrict i\forces p(i)\in \dot \q^{\eta}_i$ for $1\leq i<n$. Since $\mathcal A$ is maximal in $\p^\eta_{n+1}$, it follows that there is some $q\in\mathcal A$ that is compatible with $p\in \p^\eta_{n+1}\subseteq \p^\lambda_{n+1}$. In particular, we have just shown the following.
\begin{lemma}\label{lem:antichainLimitCase}
If $\lambda$ is a limit ordinal and $\xi<\lambda$, then $\p_n^\xi\subseteq \p_n^\lambda$ and every maximal antichain of $\p_n^\xi$ from $M_\xi$ remains maximal in $\p_n^\lambda$.
\end{lemma}
\begin{lemma}\label{lem:antichainTreeIteration}
If $\alpha<\beta<\omega_1$, then every maximal antichain  of $\p(\vec P_{\alpha},\ex{\lt\omega}\omega)$ from $M_\alpha$ remains maximal in $\p(\vec P_{\beta},\ex{\lt\omega}\omega)$.
\end{lemma}
\noindent The proof uses Lemma~\ref{lem:antichainSuccessorCase} for successor stages and follows easily for limit stages.

For $n<\omega$, let $\p_n^J=\p_n^{\omega_1}$ be constructed as all other limit stages and let $\vec \p^J=\la \p_n^J\mid n<\omega\ra$ be the corresponding $\omega$-iteration.

It will be useful for future arguments to assume that conditions in $\p_n^J$ are coded by reals.
\begin{theorem}\label{th:countableTreeIterationCCC}
The poset $\p(\vec P^J,\ex{\lt\omega}\omega)$ has the ccc.
\end{theorem}
\begin{proof}
Fix a maximal antichain $\mathcal A$ of $\p(\vec P^J,\ex{\lt\omega}\omega)$. Choose some transitive $M\prec L_{\omega_2}$ of size $\omega_1$ with $\mathcal A\in M$. We can decompose $M$ as the union of an elementary chain of countable substructures $$X_0\prec X_1\prec\cdots\prec X_\alpha\prec\cdots\prec M$$ with $\mathcal A\in X_0$. By the properties of $\diamondsuit$, there is some $\alpha$ such that $\alpha=\omega_1\cap {X_\alpha}$, $\p_n^\alpha=\p_n^J\cap X_\alpha$ for all $n<\omega$, and $S_\alpha$ codes $X_\alpha$. Let $M_\alpha$ be the transitive collapse of $X_\alpha$. Then $\p_n^\alpha$ is the image of $\p_n^J$ under the collapse and $\alpha$ is the image of $\omega_1$. Let $\bar {\mathcal A}=\mathcal A\cap X_\alpha$ be the image of $\mathcal A$.  So at stage $\alpha$ in the construction of $\vec P^J$, we chose a forcing extension $M_\alpha[G]$ of $M_\alpha$ by $\q(\vec P_{\alpha},\ex{\lt\omega}\omega)$ and let $\p^{\alpha+1}_n=(\p_n^\alpha)^*$ be constructed in $M_\alpha[G]$. By Lemma~\ref{lem:antichainTreeIteration}, $\bar {\mathcal A}$ remains maximal in all further $\p(\vec P_\beta,\ex{\lt\omega}\omega)$ for $\beta>\alpha$, and hence $\bar{\mathcal A}$ remains maximal in $\p(\vec P^J,\ex{\lt\omega}\omega)$. It follows that $\bar {\mathcal A}=\mathcal A$ is countable.
\end{proof}
\section{The generalized Kanovei-Lyubetsky Theorem}\label{sec:KanoveiLyubetskyGeneralized}
Now that we have constructed the $\omega$-iteration $\vec P^J=\la \p_n^J\mid n<\omega\ra$ and the corresponding tree iteration $\p(\vec P^J,\ex{\lt\omega}\omega)$ of Jensen's forcing, we would like to generalize the Kanovei-Lyubetsky argument of Section~\ref{sec:KanoveiLyubetsky} to prove the analogue of the ``uniqueness of generics" property for tree iterations. Namely, we will show that in a forcing extension $L[G]$ by $\p(\vec P^J,\ex{\lt\omega})$ the only $L$-generic filters for $\p_n^J$ are the restrictions of $G$ to a level $n$ node.

Suppose that $\vec P=\la \p_n\mid n<\omega\ra$ is an $\omega$-iteration of perfect posets and $H$ is a generic filter for $\p(\vec P,\ex{\lt\omega}\omega)$. Given a node $s$ on level $n$ of $\ex{\lt\omega}\omega$, let $x_s$ be the $n$-length sequence of generic reals added by $H$ on node $s$ and let $\dot x_s$ be the canonical name for $x_s$.

For the next lemma, suppose that $\vec P=\la \p_n\mid n<\omega\ra$ is an $\omega$-iteration of perfect posets that is an element of a suitable model $M$. We should think of $\vec P$ as one of the $\omega$-iterations $\vec P_\alpha$ arising at stage $\alpha$ in the construction of the $\omega$-iteration $\vec P^J$ and we should think of $M$ as the model $M_\alpha$ from that stage.
\begin{lemma}\label{le:denseToAvoidBranchesInTreeIteration}
In $M$, suppose that $\dot r$ is a $\p(\vec P,\ex{\lt\omega}\omega)$-name for an $n$-length sequence of reals such that for all nodes $s$ on level $n$ of $\ex{\lt\omega}\omega$, $$\one\forces_{\p(\vec P,\ex{\lt\omega}\omega)} \dot r\neq \dot x_s.$$ Then in a forcing extension $M[G]$ by $\q(\vec P,\ex{\lt\omega}\omega)$, for every node $s$ on level $n$ of $\ex{\lt\omega}\omega$, the set of conditions forcing the statement
\begin{center}
$\Phi(s):=$``If $\dot r$ is $M[G]$-generic for $\p^*_n$, then $\la \T^0_{s(0)},\dot \T^1_{s(1)},\ldots,\dot \T^{n-1}_{s(n-1)}\ra$ is not in the filter determined by $\dot r$."
\end{center}
is dense in $\p(\vec P^*,\ex{\lt\omega}\omega)$.
\end{lemma}
\begin{proof}
Fix a condition $f_X\in \p(\vec P^*,\ex{\lt\omega}\omega)$ and a node $d$ on level $n$ of $\ex{\lt\omega}\omega$. We need to find a condition $f'_{X'}\leq f_X$ such that $f'_{X'}\forces\Phi(d)$. Since $\U$ is dense in $\p(\vec P^*,\ex{\lt\omega}\omega)$, we can assume without loss that $f_X\in\U$. For a node $s$ on level $m$ of $X$, let $$f_X(s)=(\T_{j_{0,s}}^0\wedge S_{0,s},\ldots,\dot \T_{j_{m-1,s}}^{m-1}\wedge \dot S_{m-1,s}).$$ Let $\la j_{0,s},\ldots,j_{m-1,s}\ra=\vec j_s$, and note that we can have $\vec j_s=\vec j_{s'}$. By strengthening if necessary, we can assume that there is a node $s\in X$ with $\vec j_s=d$. Repeatedly using the construction in the proof of Lemma~\ref{lem:subtreeContainedinSiteration} on terminal nodes, we can find a determined condition $g_Y\in G$ with $g_Y(\vec j_s)=(q_{\vec j_s},F_{\vec j_s})$ such that for every node $s\in X$, some $\tau_s:m\to \ex{\lt\omega}2$ lies on $(q_{\vec j_s},F_{\vec j_s})$ with
\begin{center}
$q_{\vec j_s}(0)_{\tau_s(0)}\leq S_{0,s}$, and for all $1\leq i<m$, $(q_{\vec j_s}\restrict i)\mid(\tau_s\restrict i)\forces q_{\vec j_s}(i)_{\tau_s(i)}\leq \dot S_{i,s}$.
\end{center}
The construction gives $\tau_s$ satisfying that if $s'$ extends $s$, then $\tau_{s'}$ extends $\tau_s$.

Now we are going to construct a condition $a_A^g\in \p(\vec P,\ex{\lt\omega}\omega)$ with $a_A^g(s)=q_s$ associated to $g_Y$, satisfying the following properties:
\begin{enumerate}
\item $X\subseteq A$ and for every node $s\in X$, $q_s\leq q_{\vec j_s}\mid \tau_s$.\\
There is $X\subseteq \bar A\subseteq A$ such that:
\item For every $\sigma$ which lies on $(q_d,F_d)$, there is a node $s_\sigma\in \bar A$ such that $q_{s_\sigma}\leq q_d\mid\sigma$.
\item For every node $s$ on level $n$ of $\bar A$, $a_A^g$ forces over $\p(\vec P,\ex{\lt\omega}\omega)$ the statement:
\begin{center}
``There is $i<n-1$ such that $\dot r\restrict i$ is (ground model) generic for $\p_i$ and $\dot r(i)\notin[(a_A^g(s)(i))_{\dot r\restrict i}]$."
\end{center}
\end{enumerate}
Let $a_{\bar A}$ be a condition in $\p(\vec P,\ex{\lt\omega}\omega)$ with $X\subseteq \bar A$ be such that $a_{\bar A}(s)=q_{\vec j_s}\mid \tau_s$ for every node $s\in X$, and for every $\sigma$ which lies on $(q_d,F_d)$, there is a node $s_\sigma\in\bar A$ such that $a_{\bar A}(s_\sigma)=q_d\mid\sigma$. Fix a node $s$ on level $n$ of $\bar A$ and consider a forcing extension $M[H]$ by $\p(\vec P,\ex{\lt\omega}\omega)$. By assumption, we have $r=\dot r_H\neq x_s$. So there is $i<n-1$ such that $r\restrict i=x_s\restrict i$ and $r(i)\neq x_s(i)$. So we can strengthen $a_{\bar A}$ to a condition $a_{A'}$ forcing that there is $i<n-1$ such that $\dot r\restrict i$ is $M$-generic and $\dot r(i)\notin[(a_{A'}(s)(i))_{\dot r(i)}]$. By repeating this for all the finitely many nodes on level $n$, we obtain the required condition $a_A^g$.

Next, we can carry out the construction in the proof of Proposition~\ref{prop:strengthenDependentAssigment}, moving up the tree instead of up the iteration, to obtain a condition $\bar g_Y\leq g_Y$ with $\bar g_Y(\vec j_s)=(\bar q_{\vec j_s},\bar F_{\vec j_s})$ such that for every $s\in X$, some $\sigma$ lies on $(\bar q_{\vec j_s},\bar F_{\vec j_s})$ having $\bar q_{\vec j_s}\mid\sigma\leq a_A^g(s)$. We can also ensure that for every $\tau$ which lies on $(\bar q_d,\bar F_d)$, if $\tau$ extends $\sigma$ which lies on $(q_d,F_d)$, then $\bar q_d\mid \tau\leq a^g_A(s_\sigma)$. By density, some such $\bar g_{Y}$ must be in $G$.

By construction, $\bar g_{Y}$ forces that $f_X$ is compatible with $a_A^g$. So let $\bar f_{\bar X}\leq f_X,a_A^g$. We will be done if we can  show that $a^g_A$ forces the statement $\Phi(d)$ over $\p(\vec P^*,\ex{\lt\omega}\omega)$.

Suppose $H^*\subseteq \p(\vec P^*,\ex{\lt\omega}\omega)$ is a $V$-generic filter containing $a^g_A$ ($V=M[G]$).  Now let's suppose towards a contradiction that $p=\la \T^0_{d(0)},\dot \T^1_{d(1)},\ldots, \dot \T^{n-1}_{d(n-1)}\ra$ is in the filter determined by $r$. Thus, there is some $\sigma$ which lies on $p$ such for for all $i<n$, $r(i)$ is a branch through $(p\mid\sigma(i))_{r\restrict i}$. By construction, we have that $p\mid\sigma\leq a^g_A(s)$ for some $s$ on level $n$ of $\ex{\lt\omega}\omega$. Let $H$ be the restriction of $H^*$ to an $M$-generic filter for $\p(\vec P,\ex{\lt\omega}\omega)$, and note that $a^g_A\in H$. Thus, there is some $i<n-1$ such that $r(i)$ is not a branch through $(a^g_A(s)(i))_{r\restrict i}$, which is the desired contradiction.
\end{proof}
\begin{theorem}\label{th:uniquenessOfGenericsTreeIterationCountable}
Suppose $H\subseteq \p(\vec P^J,\ex{\lt\omega}\omega)$ is $L$-generic. If an $n$-length sequence of reals $r\in L[H]$ is $L$-generic for $\p_n$, then $r=x_s$ for some node $s$ on level $n$ of $\ex{\lt\omega}\omega$.
\end{theorem}
\begin{proof}
Let's suppose that $r$ is not one of the $x_s$ for $s$ on level $n$ of $\ex{\lt\omega}\omega$. Let $\dot r$ be a nice $\p(\vec P^J,\ex{\lt\omega}\omega)$-name for $r$ such that for all nodes $s$ on level $n$ of $\ex{\lt\omega}\omega$,
$\one\forces_{\p(\vec P^J,\ex{\lt\omega}\omega)} \dot r\neq \dot x_s$.

Choose some transitive $M\prec L_{\omega_2}$ of size $\omega_1$ with $\dot r\in M$. We can decompose $M$ as the union of a continuous elementary chain of countable substructures $$X_0\prec X_1\prec\cdots\prec X_\alpha\prec\cdots\prec M$$ with $\dot r\in X_0$. By the properties of $\diamondsuit$, there is some $\alpha$ such that $\alpha=\omega_1\cap{X_\alpha}$, $\p_n^\alpha=\p_n^J\cap X_\alpha$ for all $n<\omega$, and $S_\alpha$ codes $X_\alpha$. Let $M_\alpha$ be the transitive collapse of $X_\alpha$. Then, for every $n<\omega$, $\p_n^\alpha$ is the image of $\p_n^J$ under the collapse, and $\alpha$ is the image of $\omega_1$. The name $\dot r$ is fixed by the collapse by our assumption that we can always code conditions in $\p(\vec \p^J,\ex{\lt\omega}\omega)$ by subsets of $\omega$ and because all antichains of $\p(\vec \p^J,\ex{\lt\omega}\omega)$ are countable. So at stage $\alpha$ in the construction of $\vec P^J$, we chose a forcing extension $M_\alpha[G]$ of $M_\alpha$ by $\q(\vec P^{\alpha},\ex{\lt\omega}\omega)$ and let $\vec P_{\alpha+1}=\vec P_\alpha^*$ be constructed in $M_\alpha[G]$.

By elementarity, $M_\alpha$ satisfies that $\one_{\p(\vec P_\alpha,\ex{\lt\omega}\omega)}\forces \dot r\neq \dot x_s$ for all $s$ on level $n$ of $\ex{\lt\omega}\omega$. Thus, by Lemma~\ref{le:denseToAvoidBranchesInTreeIteration}, for every $s$ on level $n$ of $\ex{\lt\omega}\omega$, $\p(\vec \p_{\alpha+1},\ex{\lt\omega}\omega)$ has a maximal antichain $\mathcal A_s$ consisting of conditions $f_X$ forcing the statement:
\begin{center}
$\Phi(s):=$``If $\dot r$ is $M_\alpha[G]$-generic for $\p^*_n$, then $\la \T^0_{s(0)},\dot \T^1_{s(1)},\ldots,\dot \T^{n-1}_{s(n-1)}\ra$ is not in the filter determined by $\dot r$."
\end{center}
It follows, by an argument analogous to the proof of Proposition~\ref{prop:MaxAntiMForcingExtension}, that every $\mathcal A_s$ remains maximal in $\p(\vec P^J,\ex{\lt\omega}\omega)$. So let's argue that if $f_X\in \mathcal A_s$, then $f_X$ forces in $\p(\vec P^J,\ex{\lt\omega}\omega)$ that if $\dot r$ is $L$-generic, then $\la \T^0_{s(0)},\dot \T^1_{s(1)},\ldots,\dot \T^{n-1}_{s(n-1)}\ra$ is not in the filter determined by $\dot r$.

Let $\bar H^*\subseteq \p(\vec P^J,\ex{\lt\omega}\omega)$ be an $L$-generic filter containing $f_X$  and let $\bar H$ be the restriction of $\bar H^*$ to $\p(\vec P_{\alpha+1},\ex{\lt\omega}\omega)$. Let $r=\dot r_{\bar H^*}$ and suppose that it is $L$-generic for $\p_n^*$. Since $f_X\in \bar H$, it follows that $M_{\alpha}[G][\bar H]$ satisfies that $\la \T^0_{s(0)},\dot \T^1_{s(1)},\ldots,\dot \T^{n-1}_{s(n-1)}\ra$ is not in the filter determined by $r$. But then this is true in $L$ as well.

Since $H$ must meet every $\mathcal A_s$, it holds in $L[H]$ that if $r$ is $L$-generic for $\p_n^J$, then it does not meet the maximal antichain consisting of conditions $\la \T^0_{s(0)},\dot \T^1_{s(1)},\ldots,\dot \T^{n-1}_{s(n-1)}\ra$ for nodes $s$ on level $n$ of $\ex{\lt\omega}\omega$ (from level $\alpha$ of the construction), and so $r$, in fact, cannot be $L$-generic.
\end{proof}

\section{Iterating $\vec P^J$ along the tree $\ex{\lt\omega}\omega_1$}
We will now argue that the tree iteration $\p(\vec P^J,\ex{\lt\omega}\omega_1)$, where we iterate along the uncountable tree $\ex{\lt\omega}\omega_1$, shares all the key properties of the tree iteration $\p(\vec P^J,\ex{\lt\omega}\omega)$, namely it has the ccc and Theorem~\ref{th:uniquenessOfGenericsTreeIterationCountable}, concerning the uniqueness of generic filters for $\p_n^J$, continues to holds.
\begin{theorem}
The poset $\p(\vec P^J,\ex{\lt\omega}\omega_1)$ has the ccc.
\end{theorem}
\begin{proof}
Let's suppose to the contrary that there is an uncountable antichain $\mathcal A$ in $\p(\vec P^J,\ex{\lt\omega}\omega_1)$. By a $\Delta$-system argument, there must be some finite subtree $X\subseteq \ex{\lt\omega}\omega_1$ and an uncountable subset $\mathcal A'\subseteq \mathcal A$ such that for any $f_Y$ and $g_Z$ in $\mathcal A'$, $Y\cap Z=X$. Given $f_Y\in \mathcal A'$, let $f_X$ be the restriction of $f_Y$ to $X$. Clearly the collection of all such $f_X$ is an uncountable antichain in $\p(\vec P^J,\ex{\lt\omega}\omega_1)$. But then since $X$ is finite, there must be a corresponding uncountable antichain in $\p(\vec P^J,\ex{\lt\omega}\omega)$, which is impossible by Theorem~\ref{th:countableTreeIterationCCC}.
\end{proof}
\begin{theorem}
\label{th:uniquenessOfGenericsTreeIterationUncountable}
Suppose $H\subseteq \p(\vec P^J,\ex{\lt\omega}\omega_1)$ is $L$-generic. If an $n$-length sequence of reals $r\in L[H]$ is $L$-generic for $\p_n^J$, then $r=x_s$ for some node $s$ on level $n$ of $\ex{\lt\omega}\omega_1$.
\end{theorem}
\begin{proof}
Suppose that $r$ is an $n$-length sequence of reals in $L[H]$ that is $L$-generic for $\p_n^J$. Let $\dot r$ be a nice $\p(\vec P^J,\ex{\lt\omega}\omega_1)$-name for $r$. Since $\p(\vec P^J,\ex{\lt\omega}\omega_1)$ has the ccc by Theorem~\ref{th:uniquenessOfGenericsTreeIterationUncountable}, it follows that conditions in the name $\dot r$ use only countably many nodes of $\ex{\lt\omega}\omega_1$. Thus, $\dot r$ is a $\p(\vec P,\mathscr T)$-name, where $\mathscr T$ is a countable subtree of $\ex{\lt\omega}\omega_1$. We may assume that $\mathscr T$ is isomorphic to $\ex{\lt\omega}\omega$, and so $\p(\vec P,\mathscr T)$ is isomorphic to $\p(\vec P,\ex{\lt\omega}\omega)$. Let $\bar H$ be the restriction of $H$ to $\p(\vec P,\mathscr T)$, which is $L$-generic for it. Thus, $r\in L[\bar H]$, from which it follows, by Theorem~\ref{th:uniquenessOfGenericsTreeIterationCountable}, that $r=x_s$ for some $s\in\mathscr T$.
\end{proof}
\section{Symmetric models of $\ZF+\AC_\omega +\neg \DC$}\label{sec:symmetricModels}
We will construct a symmetric submodel of a forcing extension $L[G]$ by $\p(\vec P^J,\ex{\lt\omega}\omega_1)$ in which $\ZF+\AC_\omega$ holds, but $\DC$ fails. The reals of this model will be a model of second-order arithmetic in which $\Z_2+\Sigma^1_\infty$-$\AC$ holds, but $\Pi^1_2$-$\DC$ fails. Let's start with a brief discussion of the method of constructing symmetric submodels of a forcing extension, which goes all the way back to Cohen's pioneering work on forcing.

Suppose that $\p$ is a forcing notion. Recall that if $\pi$ is an automorphism of $\p$, then we can apply $\pi$ recursively to conditions in a $\p$-name $\sigma$ to obtain the $\p$-name $\pi(\sigma)$. It is not difficult to see by induction on complexity of formulas, that for every formula $\varphi$ and condition $p\in \p$, $p\forces \varphi(\sigma)$ if and only if $\pi(p)\forces \varphi(\pi(\sigma))$. Fix some group $\mathcal G$ of automorphisms of $\p$. Recall that a filter $\mathscr F$ on subgroups of a group $\mathcal G$ is \emph{normal} if whenever $g\in \mathcal G$ and $\mathcal K\in \mathscr F$, then $g\mathcal K g^{-1}\in \mathscr F$. Let's fix a normal filter $\mathscr F$ on the subgroups of $\mathcal G$. The subgroup of $\mathcal G$ fixing a particular $\p$-name $\sigma$, consisting of automorphisms $\pi$ such that $\pi(\sigma)=\sigma$, is called $\sym(\sigma)$. If $\sym(\sigma)$ is in $\mathscr F$, then we say that $\sigma$ is a \emph{symmetric} $\p$-name. We recursively define that a $\p$-name is \emph{hereditarily symmetric} when it is symmetric and all names inside it are hereditarily symmetric. Let $\HS$ be the collection of all hereditarily symmetric $\p$-names. Let $G\subseteq \p$ be $V$-generic. The \emph{symmetric model} $$N=\{\sigma_G\mid \sigma\in \HS\}$$ associated to the group of autmorphisms $\mathcal G$ and the normal filter $\mathscr F$ consists of the interpretations of all hereditarily symmetric $\p$-names. It is a standard result that $N\models\ZF$.

We would like to review now a classical construction of a symmetric model in which countable choice holds but dependent choice fails, which is due to Jensen~\cite{jensen:ACplusNotDC}. Let $\p$ be the forcing adding a set of Cohen subsets of $\omega_1$, indexed by the tree $\ex{\lt\omega}\omega_1$, with countable conditions. We will call the poset adding a Cohen subset to $\omega_1$ with countable conditions $\Add(\omega_1,1)$. So, more precisely, conditions in $\p$ are functions $f_X:X\to \Add(\omega_1,1)$, where $X$ is some countable subtree of $\ex{\lt\omega}\omega_1$, ordered so that $g_Y\leq f_X$ whenever $X\subseteq Y$ and for all $t\in X$, $g_Y(t)\leq f_X(t)$ in $\Add(\omega_1,1)$. Note that $\p$ is countably closed. Next, we need to select an appropriate group of automorphisms of $\p$.

Let $\Aut(\ex{\lt\omega}\omega_1)$ be the automorphism group of the tree $\ex{\lt\omega}\omega_1$. Every automorphism $\pi\in \Aut(\ex{\lt\omega}\omega_1)$ induces an automorphism $\pi^*$ of $\p$, so that $\pi^*(f_X)=f_{\pi\image X}$ with $f_{\pi\image X}(t)=f_X(\pi^{-1}(t))$. Let $$\mathcal G=\{\pi^*\mid \pi\in\Aut(\ex{\lt\omega}\omega_1)\}.$$ Next, we need to select an appropriate filter on the subroups of $\mathcal G$. Let's call a subtree $T$ of the tree $\ex{\lt\omega}\omega_1$ \emph{useful} if it is countable and doesn't have an infinite branch. Given a useful tree $T$, let $H_T$ be the subgroup of $\mathcal G$ consisting of all automorphisms $\pi^*$ such that $\pi$ point-wise fixes $T$. Let $\mathscr F$ be the filter on the subgroups of $\mathcal G$ generated by all such subgroups $H_T$ with $T$ a useful tree. To see that $\mathscr F$ is normal, observe that if $T$ is a useful tree and $H_T\subseteq K\in\mathscr F$, then $\pi\image T$ is useful and $H_{\pi\image T}\subseteq \pi^*K\pi^{*-1}$.

Now, let $G\subseteq\p$ be $V$-generic and let $N=\{\sigma_G\mid \sigma\in \HS\}$ be the symmetric model associated to $\mathcal G$ and $\mathscr F$. In $V[G]$, let $\T$ be the tree isomorphic to $\ex{\lt\omega}\omega_1$ whose nodes are the Cohen subsets of $\omega_1$ added by $G$. Given a node $t\in\ex{\lt\omega}\omega_1$, let $\sigma_t$ be the canonical name for the Cohen subset of $\omega_1$ added on node $t$ by $G$. Let $\dot \T=\{( {\rm op}(\sigma_s,\sigma_t),\one_{\p})\mid s\leq t\text{ in }\ex{\lt\omega}\omega_1\}$, where ${\rm op}(\sigma_s,\sigma_t)$ is the canonical $\p$-name for the ordered pair of $\sigma_s$ and $\sigma_t$. Clearly $\dot \T_G=\T$. Fix any $\pi^*\in \mathcal G$, and observe that $\pi^*(\dot \T)=\{({\rm op}(\sigma_{\pi(s)},\sigma_{\pi(t)}),\one_{\p})\mid s\leq t\text{ in }\ex{\lt\omega}\omega_1\}=\dot \T$. Also, any automorphism $\pi^*$ with $\pi(s)=s$ fixes $\sigma_s$. This shows that $\dot \T\in \HS$, and hence $\T$ is in the symmetric model $N$. Using that a symmetric name must be fixed by a group of automorphisms point-wise fixing a tree without an infinite branch, we can show that no infinite branch through $\mathcal T$ can have a symmetric name, and so $\DC$ fails in $N$. Using that $\p$ is countably closed, we can show that $\AC_\omega$ holds in $N$.

We are now ready to construct a symmetric model $N$ of a forcing extension $L[G]$ by $\p(\vec P^J,\ex{\lt\omega}\omega_1)$ with the same properties. First, observe that every $\pi\in\Aut(\ex{\lt\omega}\omega_1)$ induces an automorphism $\pi^*$ of $\p(\vec P^J,\ex{\lt\omega}\omega_1)$ so that $\pi^*(f_X)=f_{\pi\image X}$ with $f_{\pi\image X}(t)=f_X(\pi^{-1}(t))$. As above, let $\mathcal G$ be the group of automorphisms $\pi^*$ for $\pi\in \Aut(\ex{\lt\omega}\omega_1)$ and let $\mathscr F$ be the normal filter generated by the subgroups $H_T$ of $\mathcal G$, consisting of all automorphisms point-wise fixing some useful subtree $T$ of $\omega_1\ex{\lt\omega}$. Note that in the present situation the domains of conditions $f_X$ are finite subtrees $X$, while we use countable trees $T$ to determine which names are symmetric. Let $N\subseteq L[G]$ be the symmetric model determined by the group of automorphisms $\mathcal G$ and the filter $\mathscr F$. In $L[G]$, consider the tree $\mathcal T$, isomorphic to $\ex{\lt\omega}\omega_1$ whose nodes are the generic sequences of reals added by $G$. The argument we gave above generalizes in a straightforward way to show that $\dot\T$, the canonical name for $\T$, is hereditarily symmetric, and hence $\T\in N$.

Suppose $T$ is any subtree of $\omega_1\ex{\lt\omega}$. If $f_X$ is a condition in $\p(\vec P^J,\ex{\lt\omega}\omega_1)$, we will denote by $f_{X\cap T}$, the restriction of $f_X$ to nodes in $T$, i.e. $f_{X\cap T}$ has domain $X\cap T$ and $f_{X\cap T}(t)=f_X(t)$. We let $$G_T=\{f_X\in G\mid X\subseteq T\}.$$ Note that if $T$ is useful, then the canonical name for $G_T$ is fixed by all elements of the subgroup $H_T$, and therefore $G_T\in N$. Thus, each $L[G_T]\subseteq N$.

If $\sigma$ is a $\p(\vec P^J,\ex{\lt\omega}\omega_1)$-name and $H_T\subseteq \sym(\sigma)$ for a useful tree $T$, we will say that $T$ \emph{witnesses} that $\sigma$ is symmetric.
\begin{proposition}\label{prop:countableTreeDecides}
Suppose that $\sigma$ is a symmetric $\p(\vec P^J,\ex{\lt\omega}\omega_1)$-name, as witnessed by a useful tree $T$, and $f_X\forces\varphi(\sigma,\check a)$. Then $f_{X\cap T}\forces\varphi(\sigma,\check a)$.
\end{proposition}
\begin{proof}
Suppose it is not the case that $f_{X\cap T}\forces\varphi(\sigma,\check a)$. Then there is a condition $g_Y\leq f_{X\cap T}$ which forces $\neg\varphi(\sigma,\check a)$. So $g_Y$ is incompatible with $f_X$ and the incompatibility must occur on nodes outside of $T$. Let $\pi\in\Aut(\ex{\lt\omega}\omega_1)$ be an automorphism which switches these nodes in $Y\setminus T$ with some nodes outside both $T$ and $X$. Clearly $\pi^*\in H_T$, and so $\pi^*(\sigma)=\sigma$. It follows that $\pi^*(g_Y)\forces\neg\varphi(\sigma,\check a)$. But, by construction, $\pi^*(g_Y)$ is compatible with $f_X$, which is the desired contradiction. Thus, we have shown that $f_{X\cap T}\forces \varphi(\sigma,\check a)$.
\end{proof}
\begin{proposition}\label{prop:subsetsOfOrdinalsSymmetricModels}
Suppose that $\sigma$ is a symmetric $\p(\vec P^J,\ex{\lt\omega}\omega_1)$-name, as witnessed by a useful tree $T$, and $A=\sigma_G$ is a set of ordinals. Then $A\in L[G_T]$.
\end{proposition}
\begin{proof}
Let $f_X$ be some condition in $G$ forcing that $\sigma$ is a set of ordinals. Define a name $$\sigma^*=\{(\xi,g_{Y\cap T})\mid g_Y\leq f_X, g_Y\forces \xi\in\sigma\}.$$ We will argue that $f_X\forces\sigma=\sigma^*$. Let $H\subseteq \p(\vec P^J,\ex{\lt\omega}\omega_1)$ be some $L$-generic filter containing $f_X$. Suppose $\xi\in\sigma_H$. Then there is $g_Y\in H$ such that $g_Y\leq f_X$ and $g_Y\forces \xi\in \sigma$, from which it follows that  $(\xi,g_{Y\cap T})\in \sigma^*$ and $g_{Y\cap T}\in H$. So we have $\xi\in \sigma^*_H$. Next, suppose that $\xi\in\sigma^*_H$. Then there is a condition $g_Y\forces \xi\in \sigma$ such that  $(\xi,g_{Y\cap T})\in \sigma^*$ and $g_{Y\cap T}\in H$. But by Proposition~\ref{prop:countableTreeDecides}, it follows that $g_{Y\cap T}\forces\xi\in\sigma$, and so $\xi\in\sigma_H$.
\end{proof}
\begin{lemma}
$\DC$ fails in $N$.
\end{lemma}
\begin{proof}
We will argue that $\mathcal T$ does not have a infinite branch in $N$, and hence $\DC$ fails. Suppose to the contrary that $N$ has an infinite branch $b$ through $\mathcal T$, which we can view via coding as a subset of natural numbers (since it is an $\omega$-length sequence of reals). Fix a symmetric name $\dot b$ for $b$, as witnessed by a useful tree $T$. By Proposition~\ref{prop:subsetsOfOrdinalsSymmetricModels}, we can assume that the name $\dot b$ mentions only conditions with domains contained in $T$. Recall that for a node $s\in \ex{\lt\omega}\omega_1$, $x_s$ is the $L$-generic sequence of reals for $\p_{\text{len}(s)}^J$ added by $G$ on node $s$ and $\dot x_s$ is the canonical $\p(\vec P^J,\ex{\lt\omega}\omega_1)$-name for $x_s$. Let $\bar b$ be the branch through the tree $\ex{\lt\omega}\omega_1$ which corresponds to $b$ via the obvious isomoprhism. Since $T$ doesn't have infinite branches by assumption, $\bar b$ cannot be a branch through $T$. Thus, there is some natural number $n$ such that $\bar b\restrict n\subseteq T$ and $\bar b(n)=s$ is outside $T$. 

Fix a condition $$f_X\forces \dot b(n)=\dot x_s,$$ and assume without loss that $s\in X$. It is easy to see that conditions $g_Y$ having some $t\in Y\setminus (X\union T)$ with $g_Y(t)=f_X(s)$ are dense below $f_X$. So fix a condition $g_Y\in G$ with $X\subseteq Y$ and $t\in Y\setminus (X\union T)$ with $g_Y(t)=f_X(s)$. Let $\pi$ be the automorphism of the tree $\ex{\lt\omega}\omega_1$ which swaps $s$ with $t$. Let $H=\pi\image G$, which is also $L$-generic for $\p(\vec P^J,\ex{\lt\omega}\omega_1)$. Observe that $\dot b_H=b$ since the name $\dot b$ only mentioned conditions with domain in $T$. Observe also that $f_X\in H$ since it is above $\pi(g_Y)=g_{\pi\image Y}$. So it must be the case that $b(n)=(\dot x_s)_H$. But this is impossible because $(\dot x_s)_H=x_t$ and $(\dot x_s)_G=x_s$ and, by genericity, $x_s\neq x_t$.

\end{proof}
The proof of the next lemma relies mainly on the fact that $\p(\vec P^J,\ex{\lt\omega}\omega_1)$ has the ccc and uses very little else about what the conditions in $\p(\vec P^J,\ex{\lt\omega}\omega_1)$ look like. So to make the notation nicer, we will switch away from the previous convention and call conditions in $\p(\vec P^J,\ex{\lt\omega}\omega_1)$ standard names like $p$ and $q$.
\begin{lemma}
$\AC_\omega$ holds in $N$.
\end{lemma}
\begin{proof}
Suppose $F\in N$ is a countable family of non-empty sets. Let $\dot F$ be a hereditarily symmetric name for $F$ with the useful tree $S$ witnessing that $\dot F$ is symmetric, and let $q\in G$ force that $\dot F$ is a countable family of non-empty sets. We would like to build a name $\dot C\in \HS$ such that $q$ forces that $\dot C$ is a choice function for $\dot F$. We will adapt the following strategy. First, we will build a mixed name $\dot C_0\in \HS$ (over an antichain below $q$) and a useful tree $T_0$ extending $S$, witnessing that $\dot C_0$ is symmetric, such that $q\forces \dot C_0\in \dot F(0)$. Next, we will build a mixed name $\dot C_1\in\HS$ and a useful tree $T_1$ extending $T_0$, witnessing that $\dot C_1$ is symmetric, such that $q\forces \dot C_1\in\dot F(1)$. Proceeding in this fashion, we will build names $\dot C_n\in\HS$ and an increasing sequence of useful trees $T_n$ such that $q\forces \dot C_n\in \dot F(n)$. Provided we can ensure in the course of the construction that $T=\Union_{n\in\omega}T_n$ does not have an infinite branch, we will be able to build from the names $\dot C_n$ a hereditarily symmetric name $\dot C$, witnessed by $T$ to be symmetric, that is forced by $q$ to be a choice function for $\dot F$.

Let $D_0$ be the dense set below $q$ of conditions $p$ such that for some name $\dot c_p\in \HS$, $p\forces\dot c_p\in\dot F(0)$. We will thin out $D_0$ to a maximal antichain over which we can mix the names $\dot c_p$ to get the desired name $\dot C_0\in\HS$. At the beginning, we choose some condition $p_0\in D_0$ and a name $\dot c_{p_0}\in\HS$, witnessed by a useful tree $S_0'$ to be symmetric, such that $p_0\forces \dot c_{p_0}\in \dot F(0)$. By extending $S_0'$, if necessary, we can assume that $\dom(p_0)\subseteq S_0'$. Let $S_0$ be the union of $S$ and $S_0'$. Next, we choose some $p_1'\in D_0$, incompatible to $p_0$, and a name $\dot c_{p_1'}\in \HS$, witnessed by a useful tree $S_1'$, with $\dom (p_1')\subseteq S_1'$, to be symmetric, such that $p_1'\forces \dot c_{p_1'}\in \dot F(0)$. At this point, it is not a good idea to union up $S_0$ and $S'_1$ because if we keep doing this we might end up with an infinite branch in the final tree. So instead, we first consider those nodes in $S_0\setminus S$ which have some new node of $S_1'$ sitting directly above them, meaning that we are extending an old branch that was already extended once from $S_0$. Call these nodes in $S_0\setminus S$ \emph{bad} and the nodes of $S_1'$ sitting directly above them \emph{bad successors}. Let $\pi$ be an automorphism of the tree $\ex{\lt\omega}\omega_1$ which moves every bad node $t$ (and its bad successors along with it) to some node outside $S_0$. Since $S_0$ is countable and each node in $\ex{\lt\omega}\omega_1$ has $\omega_1$-many successors, there are many nodes to which the nodes $t$ can move. Since $\pi$ fixes $S$, we have $\pi^*\in H_S$. So we have $\pi^*(p_1')\forces \pi^*(\dot c_{p_1'})\in \dot F(0)$. We can safely union up $S_0$ and $\pi\image S_1'$ without growing any branches in $S_0$ that have extended beyond $S$. But the problem now is that $\pi^*(p_1')$ may no longer be incompatible with $p_0$, since the incompatibility could have come from some node $t$ that got moved. This, however, can be easily remedied by adding back to $\pi^*(p_1')$ one of the nodes from $S_0$ which caused the incompatibility. Let $p_1$ be the condition $\pi^*(p_1')$ together with possibly one extra node from $S_1'\cap S_0$, let $\dot c_{p_1}=\pi^*(\dot c_{p_1'})$, and let $S_1$ be the union of $S_0$ and $\pi\image S_1'$. Using that the filter $\mathscr F$ is normal, it is not difficult to see that the image of a name $\sigma$ in $\HS$ under an automorphism $\pi^*$ from $\mathcal G$ is also a name in $\HS$, and moreover if some useful tree $T$ witnesses that $\sigma$ is symmetric, then $\pi\image T$ witnesses that $\pi^*(\sigma)$ is symmetric. So $\dot c_{p_1}\in \HS$ and $S_1$ witnesses that it is symmetric. Finally, note that $\dom (p_1)\subseteq S_1$. In the next step, we choose some condition in $D_0$ incompatible with both $p_0$ and $p_1$ together with some name in $\HS$ it forces to be in $\dot F(0)$. We use an automorphism from $H_S$ to transform the witnessing tree into a tree that can be unioned up with $S_1$ without growing any branches by moving over nodes (outside $S$) that have new nodes sitting directly above them. We continue this process transfinitely for as long as we can find a condition in $D_0$ that is incompatible to all previously chosen conditions. At limit stages, we simply take the union of the increasing sequence of trees constructed up to that stage. The process must terminate by some countable ordinal stage $\alpha$ since $\p(\vec P^J,\ex{\lt\omega}\omega_1)$ has the ccc. Thus, the final union tree $T_0=\Union_{\xi<\alpha}S_\xi$ must be countable and by construction, since we never allow any branch to grow more than once, $T_0$ cannot have an infinite branch.

Let's argue that the tree $T_0$, because it contains $\Union_{\xi<\alpha}\dom(p_\xi)$, witnesses that the mixed name $\dot C_0$ of the names $\dot c_{p_\xi}$ for $\xi<\alpha$ over the antichain $\la p_\xi\mid \xi<\alpha\ra$ is symmetric. Recall that $$\dot C_0=\Union_{\xi<\alpha}\{( \tau,r)\mid r\leq p_\xi,\,r\forces \tau\in \dot c_{p_\xi},\tau\in\text{\dom}(\dot c_{p_\xi})\}.$$ Fix an automorphism $\pi$ point-wise fixing $T_0$. It suffices to argue that whenever $(\tau,r)\in \dot C_0$, then $(\pi^*(\tau),\pi^*(r))\in \dot C_0$. So suppose $(\tau,r)\in\dot C_0$ and fix $p_\xi$ witnessing this. Since $r\leq p_\xi$, it follows that $\pi^*(r)\leq p_\xi$; since $r\forces \tau\in \dot c_{p_\xi}$, it follows that $\pi^*(r)\forces\pi^*(\tau)\in \dot c_{p_\xi}$; and finally, since $\dot c_{p_\xi}$ is symmetric and $\tau\in\text{dom}(\dot c_{p_\xi})$, it follows that $\pi^*(\tau)\in\text{dom}(\dot c_{p_\xi})$. Now observe that, since each $\dot c_{p_\xi}\in \HS$, we have $\dot C_0\in \HS$.

Next, we let $D_1$ be the dense set below $q$ of conditions $p$ such that for some name $\dot c_p\in \HS$, $p\forces\dot c_p\in\dot F(1)$. We thin out $D_1$ to a maximal antichain over which we can mix to obtain the desired symmetric name $\dot C_1$ by the same process as above building the new trees as extensions of $T_0$, where branches that have extended beyond $S$ are no longer allowed to grow. Once, $\dot C_1$ and $T_1$ have been obtained, we proceed to obtain $\dot C_2$ and $T_2$ extending $T_1$, and continue in this manner for $\omega$-many steps. The final tree $T=\Union_{n<\omega}T_n$ is still countable, cannot have an infinite branch by construction, and witnesses that the canonical name for the sequence of the $\dot C_n$ is symmetric.
\end{proof}
\section{A second-order arithmetic model of $\Sigma^1_{\infty}$-$\AC$+$\neg\Pi^1_2$-$\DC$}\label{sec:SOArithmetic}
We will now argue that the reals of the symmetric submodel $N$ of the forcing extension $L[G]$ by $\p(\vec P^J,\ex{\lt\omega}\omega)$, constructed in Section~\ref{sec:symmetricModels}, is a  second-order arithmetic model of $\Z_2$ together with $\Sigma^1_\infty$-$\AC$ in which $\Pi^1_2$-$\DC$ fails.
\begin{theorem}\label{th:SOArithmeticModelACnegDC}
There is a $\beta$-model of second-order arithmetic $\Z_2$ together with $\Sigma^1_\infty$-$\AC$ in which $\Pi^1_2$-$\DC$ fails.
\end{theorem}
\begin{proof}
Let $G\subseteq \p(\vec P^J,\ex{\lt\omega}\omega_1)$ be $L$-generic and let $N\subseteq L[G]$ be the symmetric submodel of $\ZF+\AC_\omega+\neg \DC$ constructed in Section~\ref{sec:symmetricModels}. Let $\M$ be the model of $\Z_2$ whose collection of sets is $P(\omega)^N$. Since $\AC_\omega$ holds in $N$, it immediately follows that $\Sigma^1_{\infty}$-$\AC$ holds in $\M$. Thus, it remains to show that $\Pi^1_2$-$\DC$ fails in $\M$. The result will follow if we can show that the tree $\T$ is $\Pi^1_2$-definable in $\M$. By Theorem~\ref{th:uniquenessOfGenericsTreeIterationUncountable}, a sequence $\la r_0,\ldots,r_{n-1}\ra$ of reals is in the domain of $\T$ if and only if it is $L$-generic for $\p_n^J$. So we need to see that the statement $\la r_0,\ldots,r_{n-1}\ra$ is $L$-generic for $\p_n^J$ is $\Pi^1_2$-expressible in $\M$. This will suffice because the sequences in $\T$ are ordered simply by extension.

What we would like to say is that all $L_\alpha$, with $\alpha$ a limit ordinal, should satisfy the following. Given an ordinal $\beta$ we check whether stage $\beta$ was nontrivial in the construction of $\vec P^J=\la \p_n^J\mid n<\omega\ra$. In this case, we check whether $\la r_0,\ldots,r_{n-1}\ra$ generates an $M_\beta$-generic filter for $\p_n^\beta$. Let $H_0$ consist of all trees in $\p_1^\beta$ which have the branch $r_0$. If $H_0$ is not $M_\beta$-generic, we are done. Otherwise, let $H_1$ consist of all conditions $p\in\p_2$ such that $p(0)\in H_0$ and $r_1$ is a branch through $p(1)_{H_0}$. If $H_1$ is not $M_\beta$-generic, we are done. Otherwise, we continue. So $\la r_0,\ldots,r_{n-1}\ra$ is $\p_n^J$-generic if whenever $\alpha$ is a limit ordinal, then $L_\alpha$ satisfies that for every nontrivial stage $\beta$ in the construction of $\p_n^J$, $\la r_0,\ldots,r_{n-1}\ra$ generates an $M_\beta$-generic filter. Now it is not difficult to check that the complexity of this statement is $\Pi^1_2$, because it says that for every set $X$, if $X$ codes a limit ordinal $\alpha$, then there is another set $Y$ coding $L_\alpha$ and a set coding the filters $H^\beta$ needed to verify for every nontrivial stage $\beta\in L_\alpha$ that $H^\beta$ is $M_\beta$-generic. The statement that $X$ codes an ordinal is $\Pi^1_1$ and the remaining statements are $\Sigma^1_1$.
\end{proof}
\section{Reflection can fail in models of $\ZFC$ without powerset}\label{sec:reflectionPrinciple}
Let's call the \emph{Reflection Principle} the statement that for every set $a$ and formula $\varphi(\bar x,a)$, there is a transitive set model $M$ containing $a$ such that $\varphi(\bar x,a)$ is absolute between $M$ and the universe. By the \Levy-Montague reflection, $\ZFC$ implies the Reflection Principle, namely every formula is reflected by some $V_\alpha$. It is natural to wonder whether the Reflection Principle holds in models of $\ZFC^-$, which may not have the $V_\alpha$-hierarchy.\footnote{In the absence of powerset, the axiom of choice, defined as the existence of choice functions, is not equivalent the assertion that every set can be well-ordered \cite{zarach:unions_of_zfminus_models}. Here the theory $\ZFC^-$ is assumed to include the assertion that every set can be well-ordered.} It has been suspected for a long time that the Reflection Principle can fail in such models. The question was first asked by Zarach in \cite{Zarach1996:ReplacmentDoesNotImplyCollection}, and considered again in \cite{zfcminus:gitmanhamkinsjohnstone}. We will show that the Reflection Principle fails in the model $H_{\omega_1}^N$, where $N$ is the symmetric submodel of the forcing extension $L[G]$ by $\p(\vec P^J,\ex{\lt\omega}\omega_1)$ that we constructed in Section~\ref{sec:symmetricModels}.

The argument requires first seeing that the Reflection Principle is equivalent to a version of $\DC$ for definable relations. Following \cite{zfcminus:gitmanhamkinsjohnstone}, let's define that the \emph{Dependent Choice Scheme}, abbreviated $\DC$-scheme, asserts for every formula $\varphi(x,y,z)$ and parameter $a$ that if for every $x$, there is $y$ such that $\varphi(x,y,a)$ holds, then there is an $\omega$-sequence $\la x_n\mid n<\omega\ra$ such that for all $n$, $\varphi(x_n,x_{n+1},a)$ holds; in other words, if a definable relation has no terminal nodes, then we can make $\omega$-many dependent choices according to it.
\begin{lemma}[\cite{zfcminus:gitmanhamkinsjohnstone}]\label{lem:RSequivDCScheme}
The Reflection Principle is equivalent over $\ZFC^-$ to the $\DC$-scheme.\footnote{This equivalence holds regardless of whether we include the existence of choice functions or well-orderings as the choice axiom in our theory.}
\end{lemma}
\begin{proof}
First, let's assume that the Reflection Principle holds. Fix a relation $\varphi(x,y,a)$ without terminal nodes. Let $M$ be a transitive set which reflects $\varphi(x,y,a)$ and let $R$ be the set relation on $M$ derived from $\varphi(x,y,a)$. Now fix a well-ordering $W$ of $M$ and use it to define an $\omega$-sequence of choices according to $R$.

Next, suppose that the $\DC$-scheme holds. Fix a formula $\varphi(\bar x,a)$. Observe that given any set $A$, using collection, we can argue that there is a set $\bar A\supseteq A$ which is closed under existential witnesses for subformulas of $\varphi(\bar x, a)$ with parameters from $A$. By taking the transitive closure we can assume that $\bar A$ is transitive. So consider the definable relation $R$ which says that $A$ is related to $\bar A$, whenever $A\subseteq\bar A$, $\bar A$ is transitive, and $\bar A$ is closed under existential witnesses for subformulas of $\varphi(\bar x,a)$ with parameters from $A$. We just argued that $R$ has no terminal nodes. Thus, we can make a sequence $\la M_n\mid n<\omega\ra$ of dependent choices in $R$. But then clearly $M=\Union_{n<\omega}M_n$ reflects $\varphi(\bar x,a)$.
\end{proof}
\begin{theorem}
The theory $\ZFC^-$ does not imply the Reflection Principle.
\end{theorem}
\begin{proof}
Let $M=H_{\omega_1}^N$, where $N$ is the $\ZF$-model constructed in Section~\ref{sec:symmetricModels}. Clearly $M\models\ZF^-$ and since $N\models\AC_{\omega}$, choice holds in $M$, so $M\models\ZFC^-$. Since $\T$ is definable in $M$ (by the argument given in the proof of Theorem~\ref{th:SOArithmeticModelACnegDC}), it follows that the $\DC$-scheme fails in $M$. Thus, by Lemma~\ref{lem:RSequivDCScheme}, the Reflection Principle fails as well.
\end{proof}
\section{Open Questions}
We end with some open questions related to the results of this article. 

To produce our $\beta$-model of $\Z_2+\Sigma^1_\infty$-$\AC$ in which $\Pi^1_2$-$\DC$ fails we used the consistency of (a fragment of) $\ZFC$. 
\begin{question}
Does the consistency of $\Z_2 + \Sigma^1_\infty\text{-}\AC + \neg\Pi^1_2\text{-}\DC$ follow just from the consistency of $\Z_2$? Does this implication hold with "consistency" replaced by "existence of a $\beta$-model"?
\end{question}
The third author showed in \cite{kanovei:ACnotDC} that it is consistent to have a model of $\Z_2+\Sigma^1_\infty$-$\AC+\Pi^1_n$-$\DC$ in which $\Pi^1_{n+1}$-$\DC$ fails. The construction used countable-support iterations of Jensen's forcing.
\begin{question}
Can we obtain a new proof of the result using techniques of this article involving finite iterations? 
\end{question}
Finally, we showed that the Reflection Principle can fail in a model of $\ZFC^-$ in which every set is hereditarily countable. 
\begin{question}\label{ques:KM}
Can the Reflection Principle fail in a model of $\ZFC^-$ with a largest cardinal $\kappa$ that is inaccessible in the model?
\end{question}
The theory $\ZFC^-$ together with the assertion that there is a largest cardinal $\kappa$ that is inaccessible is bi-interpretable with the second-order set theory Kelley-Morse together with an appropriate version of the choice scheme. Kelley-Morse is the set-theoretic analogue of $\Z_2$ in that it contains the comprehension scheme for all second-order assertions. The strategy for answering Question~\ref{ques:KM} involves defining a version of Jensen's forcing for an inaccessible cardinal $\kappa$ and carrying out the rest of the construction to produce a symmetric model $N$ such that $(V_\kappa^N,V_{\kappa+1}^N)$ is a model of Kelley-Morse together with the choice scheme in which the dependent choice scheme fails. We will undertake this project in an upcoming article.
\bibliography{ACnotDC}
\bibliographystyle{alpha}
\end{document}